\documentclass[a4paper,10pt]{amsart}
\usepackage{amsthm,amsmath,amssymb,amsfonts}
\usepackage[T1]{fontenc}
\usepackage[utf8]{inputenc}
\usepackage{verbatim}
\usepackage{marvosym}
\usepackage{stackrel}
\usepackage{graphicx}
\usepackage{eucal}
\usepackage{bbold}
\usepackage[svgnames]{xcolor}
\definecolor{blue(munsell)}{rgb}{0.0, 0.5, 0.69}
\usepackage{lipsum}
\usepackage{svg}
\usepackage{bbm}
\usepackage{caption}
\usepackage{enumitem}
\usepackage{mathtools}
\usepackage{epigraph}
\usepackage{color}
\usepackage{microtype}
\usepackage{relsize}
\usepackage{amsfonts}
\usepackage{adjustbox}
\usepackage{subfig}
\usepackage{framed}
\usepackage{hyperref}
\hypersetup{
    colorlinks = true,
    linkbordercolor = {red},
   	linkcolor ={blue},
	anchorcolor = {pink},
	citecolor =  {red},
	filecolor = {blue},
	menucolor = {blue},
	runcolor =  {blue},
	urlcolor = {blue},
}
\usepackage{cleveref}
\usepackage{trimclip}

\DeclareFontFamily{U}{min}{}
\DeclareFontShape{U}{min}{m}{n}{<-> udmj30}{}
\newcommand{\yo}{\!\text{\usefont{U}{min}{m}{n}\symbol{'210}}\!}

\usepackage{stmaryrd}
\usepackage{tikz}
\usepackage{tikz-cd}
\usepackage{quiver}

\theoremstyle{definition}
\newtheorem{thm}{Theorem}[subsection]
\newtheorem*{thm*}{Theorem}
\newtheorem{prop}[thm]{Proposition}
\newtheorem*{prop*}{Proposition}

\newtheorem{defn}[thm]{Definition}
\newtheorem{war}[thm]{Warning}
\newtheorem*{war*}{Warning}
\newtheorem{rem}[thm]{Remark}
\newtheorem{constr}[thm]{Construction}
\newtheorem{exa}[thm]{Example}

\newtheorem{notat}[thm]{Notation}
\newtheorem{disc}[thm]{}

\Crefname{rem}{Remark}{Remarks}
\Crefname{defn}{Definition}{Definitions}
\Crefname{prop}{Proposition}{Propositions}
\Crefname{constr}{Construction}{Constructions}
\Crefname{thm}{Theorem}{Theorems}
\Crefname{exa}{Example}{Examples}
\Crefname{disc}{Discussion}{Discussions}
\Crefname{war}{Warning}{Warnings}

\newcommand{\Acc}{\mathsf{Acc}}

\newcommand{\Set}{\mathrm{Set}}

\title{The geometry of Coherent topoi and Ultrastructures}

\author{Ivan Di Liberti}

\address{
Ivan \textsc{Di Liberti}: \newline
Department of Mathematics\newline
Stockholm University\newline
Stockholm, Sweden\newline
\href{mailto:diliberti.math@gmail.com}{\sf diliberti.math@gmail.com}
}
\thanks{The author was supported by the Swedish Research Council (SRC, Vetenskapsrådet) under Grant No.~2019-04545. The research has received funding from Knut and Alice Wallenbergs Foundation through the Foundation's program for mathematics.}

\begin{document}

\maketitle

\begin{abstract}
We show that coherent topoi are right Kan injective with respect to flat embeddings of topoi. We recover the ultrastructure on their category of points as a consequence of this result. We speculate on possible notions of ultracategory in various arenas of formal model theory.
 
  
\end{abstract}

   {
   \hypersetup{linkcolor=black}
   \tableofcontents
   }

\section*{Introduction}

One of the leitmotivs of topos theory is that we can use topological intuition as a guiding principle to devise correct definitions. The motivation for this paper is to clarify the notion of ultrastructure and ultracategory via a topological approach. Ultrastructures were defined by Makkai \cite{makkai1987stone} to condense the main properties of the category of models of a first order theory. He successfully used this technology to provide a reconstruction theorem for first order logic that goes under the name of \textit{strong conceptual completeness}.

\begin{thm*}[Conceptual completeness]
Let $\mathcal{G}$ be a small pretopos, then $\mathcal{G}$ is equivalent to the category of \textit{ultrafunctors} $f: \mathsf{Mod}(\mathcal{G}) \to \Set$,
\[\mathcal{G} \simeq \mathsf{Ult}(\mathsf{Mod}(\mathcal{G}),\Set).\]

\end{thm*}

This result falls under the umbrella of Stone-like dualities, as many others celebrated results like Gabriel--Ulmer duality, Gabriel--Rosenberg reconstruction theorem, Isbell duality, Tannaka (and Morita) reconstruction and many others. The general idea behind an ultrastructure is to account for the construction of ultraproducts of models. Indeed, given a coherent/first order theory $\mathbb{T}$, a set $X$, and an ultrafilter $U \in \beta(X)$, one can define the functor, \[\int_X (-)dU: \mathsf{Mod}(\mathbb{T})^X \to \mathsf{Mod}(\mathbb{T}).\] 

Such functor takes an $X$-indexed family of models and computes the ultraproduct of those along the ultrafilter $U$. Our nonstandard notation above follows Lurie's \cite{lurieultracategories}.

Despite delivering a very satisfactory theorem, the notion of ultrastructure seemed quite complicated -- if not ad-hoc -- and it was never clear whether it was the correct or definitive notion. The first attempt of providing a more conceptual framing to understand ultrastructures and ultracategories was given by Marmolejo in his PhD thesis \cite{marmolejo1995ultraproducts}. He was probably the first to trace some topological aspects of the construction of ultraproducts, introducing the notion of \textit{Łoś category}. We will see that the blueprint of our approach is essentially the same of Marmolejo's, even though we have a quite different way of encoding the same idea. 

In more recent years Lurie \cite{lurieultracategories} revisited the notion of ultracategory, proposing a morally similar, but technically different notion of ultrastructure, ultracategory and ultrafunctor with respect to Makkai's one. Both Makkai's and Lurie's notions are justified by the fact that they manage to deliver the most compelling theorems of this theory. Yet, none of these notions appears \textit{definitive} when read or encountered for the first time for several reasons. The main one being that the definition of ultracategory is in both cases very heavy, and comes together with axioms whose choice seems quite arbitrary; and indeed the two authors make different choices.

\textbf{In this paper} we study the topological properties of coherent topoi with respect to \textit{flat embeddings}, and we let the notion of ultrastructure emerge naturally from general considerations on the topology of flat embeddings. Our aim is to provide notions of ultrastructure and ultracategory that seem correct when instantiated in the appropriate formal model theory. 



\subsection*{Main results and structure of the paper}

In \Cref{sec1} we start by defining flat morphisms of topoi and in a sense represents a topos theoretic reinterpretation of some seminal papers by Escardo \cite{escardo1997injective,escardo1998properly,escardo1999semantic}. The notion is inspired by an old result of Joyal \cite[III.1.11]{johnstone1986stone} which proved to be crucial in the theory of coherent locales. Next, we briefly recall the theory of (right) Kan injectivity from \cite{dlsolo}. In a nutshell, this amounts to a variation of the classical theory of injectivity where a $2$-dimensional aspect is kept into account. The main result of the section is the following proposition,

\begin{thm*}[\ref{mainth}]
Coherent topoi are right Kan injective with respect to flat embeddings.
\end{thm*}

 Of course, at this point it is natural to ask whether all topoi which are right Kan injective with respect to flat embeddings are (retracts of) coherent (ones), and we discuss this in \Cref{characterisation}. We have no definitive answer to this question, besides the proposition below.

\begin{prop*}[\ref{char2}]
Let $\mathcal{E}$ be a topos which is right Kan injective with respect to flat embeddings. Then $\mathcal{E}$ admits a geometric surjection from a spatial coherent topos. (In particular it has enough points).
\end{prop*}

At the end of the section we recall the notion of \textit{coherent} geometric morphism (\Cref{defnspartan}) and we show that they are morphisms of right Kan injectives with respect to flat embeddings (\Cref{moderatesarekan}), so that we obtain an inclusion from the $2$-category of coherent topoi into the $2$-category of right Kan injectives with respect to flat morphisms,
\[\mathsf{CohTopoi} \hookrightarrow \mathsf{RInj}(\mathsf{Emb}_\flat).\]


In \Cref{sec2} we concentrate on a special class of flat embeddings of topoi. For $X$ a set, call $i_X : X \to \beta(X)$ the inclusion of $X$ in its space of ultrafilters. Then, the induced geometric morphisms,
\[i_X: \mathsf{Set}^X \to \mathsf{Sh}(\beta(X))\] is a flat embedding of topoi. We collect these morphisms in a class $\mathsf{Emb}_\beta$ and we call \textit{$\beta$-complete} a topos which is right Kan injective with respect to those. The main point of the section is to convince the reader that when a topos is $\beta$-complete, its category of points $\mathsf{pt}(\mathcal{E})$ comes equipped with a canonical ultrastruture. 

Being right Kan injective has both an existence part (\Cref{emergentultrastructures})  and universality part (\Cref{2dimensional}) to it. These two aspects mark the category of points of the topos and must be accounted as the two main ingredients for our notion of ultrastruture,

    \[  \int_X(-)d(-): \mathsf{pt}(\mathcal{E})^X \times \beta(X) \to \mathsf{pt}(\mathcal{E}). \]

\Cref{sec3} starts with a brief digression on formal model theory (\cite{thlo}), that is the study of $2$-categories whose objects look like categories of models of first order theories. Besides two arenas of formal model theory that were already studied by the author in \cite{thgeo,thlo}, namely accessible categories with directed colimits $(\mathsf{Acc}_\omega)$ and bounded ionads $(\mathsf{BIon})$, the paper introduces the notion of accessible profile $(\mathsf{Acc}_\omega^\mathsf{Sp})$, on which we shall comment in a moment.

\subsubsection*{$\mathsf{Acc}_\omega$, accessible categories with directed colimits}
For accessible categories with directed colimits,  the most classically developed approach to the topic (\cite{Makkaipare} and \cite{adamek_rosicky_1994}), we give a definition of ultracategory (\Cref{def:ultracategory}) which simplifies and conceptualizes Lurie's notion. 

\subsubsection*{$\mathsf{Acc}_\omega^\mathsf{Sp}$, accessible profiles}
The next arena of formal model theory is an enhancement of the previous one, where we keep track of the models of a theory in a broader class of spacial topoi. For $\mathcal{E}$ a Grothendieck topos, one can define a prestack of over the category $\mathsf{Sp}$ of topological spaces by simply putting,
\begin{align*}
    \mathcal{E}^{\bullet} : & \mathsf{Sp}^\circ \to \mathsf{Cat} \\
    & X \mapsto \mathsf{Topoi}(\mathsf{Sh}(X),\mathcal{E})
\end{align*}

Of course, for $X$ the terminal space, we recover the accessible category of $\mathsf{Set}$-points of the topos. With this idea in mind, we introduce the notion of \text{accessible profile}. These are $\mathsf{Sp}$-indexed accessible categories with directed colimits $\mathcal{A}^\bullet: \mathsf{Sp}^\circ \to \mathsf{Acc}_\omega$, and are further motivated by the reflections of Marmolejo on ultracategories and some of Lurie's proof techniques.

For an accessible profile $\mathcal{A}^\bullet$, we define the notion of \textit{topos of coordinates} $\Theta(\mathcal{A}^\bullet)$ (\Cref{coordinates}), such gadget offers a test-topos for reconstruction-like results (\Cref{reprensentation}). The construction of the topos of coordinates yields a general adjunction between topoi and accessible profiles (\Cref{newadjunction}). Accessible profiles offer a solid foundation for formal category theory. To exemplify that, we introduce the notion of ultraprofile and we provide the following result.
    
\begin{prop*}[\ref{ultaprofilebuono}]
Let $\mathcal{E}$ be a topos. $\mathcal{E}$ is $\beta$-complete if and only if its profile of points $\mathbb{pt}^\bullet(\mathcal{E})$ is an ultraprofile.
\end{prop*}

\subsubsection*{$\mathsf{BIon}$, ionads}
A similar treatment is given for Garner's bounded ionads \cite{ionads,thgeo,thlo}, for which we introduce the notion of ultraionad and we prove the following theorem.

\begin{prop*}[\ref{ionadsaregood}]
\begin{itemize}
\item[]
    \item A topos $\mathcal{E}$ is $\beta$-complete if and only if its ionad of points $\mathbb{pt}(\mathcal{E})$ is an ultraionad.
    \item A ionad $(\mathcal{A}, \mathsf{Int})$ is an ultraionad if and only if its topos of opens $\mathbb{O}(\mathcal{A})$ is a $\beta$-complete topos.
\end{itemize} 
\end{prop*}

\begin{war*}
Throughout the paper we intensely deploy both the general theory of Kan extensions in a $2$-category and the more specific theory of pointwise Kan extensions in $\mathsf{Cat}$. We refer to \cite[Chap. 6]{riehl2017category} as a general and solid reference for all the results we apply. The Appendix of  \cite{liberti2019codensity} can be used as a covenient cheat sheet for almost everything we need.
\end{war*} 
\section{Flat embeddings and coherent topoi} \label{sec1}

\subsection{Flat embeddings}

\begin{defn}\label{defn:flat-maps}
A geometric morphism $f : \mathcal{F} \to \mathcal{G}$ is \textit{flat} if its direct image $f_* : \mathcal{F} \to \mathcal{G}$ preserves finite coproducts and regular epis and it is \textit{strongly flat} if  its direct image $f_* : \mathcal{F} \to \mathcal{G}$ preserves finite colimits.
\end{defn}

\begin{disc}
The notion of flat morphism appeared for the first time in the theory of locales \cite[III.1.11]{johnstone1986stone}, and is apparently due to Joyal: a morphism of locales is \emph{flat} if its direct image preserves finite joins. It was later studied in few other papers by Johnstone \cite{johnstone1984wallman} and Isbell \cite{isbell1988flat}. Since the very beginning, the notion was linked to that of coherent locale, and the connection has been made more and more precise in recent years \cite{escardo1999semantic, carvalho2017kan}. This paper benefited a lot from the topological intuition provided by this literature on the theory of locales and we will soon revisit many of these classical results in a more topos-theoretic fashion.
\end{disc}

\begin{war}[Notions of flatness] \label{hier}
Originally, this paper was focusing only on the notion of strong flatness and had conveyed the two notions by mistake. Besides the two presented above, categorifying the localic case, there are many possible candidates for the notion of flat morphism of topoi: for instance, one could suggest a weaker and more \textit{propositionally inspired} one, requesting $f_*$ to preserve finite joins of subobjects. This latter notion encompasses, for instance, all morphism of localic topoi $f: \mathsf{Sh}(\mathcal{L}) \to \mathsf{Sh}(\mathcal{Q})$ coming from flat maps of locales $\mathcal L \to \mathcal Q$ \cite[Thm.\ 2.1, Prop.\ 3.2 and 3.4]{marmolejo2005locale}. Yet, it is flatness in the sense of \Cref{defn:flat-maps} which ends up having the most poignant examples for the sake of this paper; the notion of strong flatness also remains of some interest, as it will feature in later papers (see \cite{dilibertiLogicConcepts2category2025} and \cite[Ex.\ 3.6.3]{dilibertiConceptualCompletenessSubgeometric2026}). 
\end{war}

\begin{disc} \label{generalities2}
In the literature on topos theory, flat geometric morphisms only appeared twice. Once in the PhD thesis of Marmolejo and its subsequents papers \cite{marmolejo1995ultraproducts,marmolejo2005locale,marmolejo1998continuous}. There, flat morphisms go unnamed, even though they play a central role in the whole thesis. Marmolejo says that a continuous function between topological spaces is \textit{ultrafinite} if its associated morphism of localic topoi is flat in our sense. A geometric characterization of ultrafinite continuous functions can be found in \cite{marmolejo2005locale}. We will come back to Marmolejo's approach later in the paper as it is indeed essential for us.

The other appearance of a flat geometric embedding in topos theory is in \cite[IX.10.4 and the discussion above]{sheavesingeometry}. While it is evident that the authors were informed of the localic interpretation of flatness, there is no contextualization in Moerdijk's and Mac Lane's work. All in all, this paper represents the first time the notion of flat geometric morphism \textit{of topoi} is formally introduced and studied.
\end{disc}

\begin{war}
Despite the name, the notion of a flat geometric morphism is unrelated -- at least to our knowledge -- to that of a \emph{flat functor}, key in topos theory. We do not entirely understand the motivation that led to this choice even in the theory of locales, but will stick to it both for the lack of a better alternative and to be consistent with the existing literature on the topic.
\end{war}

\begin{exa}[Every locale is a flat sublocale of a coherent one] \label{exampleindcompl}
Johnstone observes that every locale admits a flat geometric embedding into a coherent locale \cite[III.1.11]{johnstone1986stone}. This is given by the embedding of a locale $\mathcal{L}$ into its frame of ideals, \[i_{\mathcal{L}}: \mathcal{L} \hookrightarrow \mathsf{Idl}(\mathcal{L}).\]
\end{exa}

\begin{exa} \label{cameo}
\looseness=-1
Let us report on the example discussed by Mac Lane and Moerdijk from \cite[IX.10.4 and the discussion above]{sheavesingeometry}, both for the sake of discussing an example, and because it will turn out to be useful for our treatment. Let $B$ be a complete boolean algebra, being a locale we can apply the construction above and embed it in its frame of ideals, $i_{\mathcal{B}}: \mathcal{B} \hookrightarrow \mathsf{Idl}(\mathcal{B}).$
In the special case of a complete boolean algebra, this is the same as the frame of opens of its Stone space, as Mac Lane and Moerdijk observe. Then, the induced geometric embedding of localic topoi is flat,
\[i_{\mathcal{B}}: \mathsf{Sh}(\mathcal{B}) \hookrightarrow \mathsf{Sh}(\mathsf{Idl}(\mathcal{B})).\]

Mac Lane and Moerdijk do not show the flatness of this geometric morphism as they only rely on the fact that $i_*$ preserve finite epimorphic families. The only available proof is due to \cite[Cor. 2.3]{marmolejo2005locale}. An essentially identical proof appears in \cite{marmolejo1995ultraproducts}, but it is much harder to read. This is an example of a flat morphism of locales whose sheafification is indeed a flat geometric morphism in the sense of \Cref{defn:flat-maps}.
\end{exa}

Mac Lane and Moerdijk use this observation to show that every coherent topos has enough points, implementing a variation of Joyal's lemma which started this theory. We will generalise their result in a couple of subsections, but we shall start from recalling a polished version of Joyal's lemma and comment it for the sake of readability.

\begin{thm} \label{paradigm}
A coherent locale $\mathcal{C}$ is injective with respect to flat embeddings of locales.

\[\begin{tikzcd}
	{\mathcal{L}} & {\mathcal{C}} \\
	{\mathcal{L'}}
	\arrow["f", from=1-1, to=1-2]
	\arrow["i"', hook, from=1-1, to=2-1]
	\arrow[dashed, from=2-1, to=1-2]
\end{tikzcd}\]
\end{thm}

\begin{disc}
The theorem above is the most paradigmatic of this theory and shows clearly the connection between coherent locales and flat embeddings via injectivity. A more precise account on the topic which takes into account the $2$-dimensional aspects of the theory and connects it to Kan injectivity is provided in \cite{carvalho2017kan}. Because every locale admits a flat embedding into a coherent one, it follows that  stably locally compact locales (retracts of coherent locales \cite[C4.1]{Sketches}) are precisely the locales (right) Kan injective to flat embeddings (\cite[Rem. 4.3]{carvalho2017kan}),

\[\mathsf{SLC} = \mathsf{RInj(FlatEmb)}.\]
\end{disc}

\begin{disc}
Mac Lane and Moerdijk prove a topos-theoretic version of \Cref{paradigm}  (\cite[IX.11.1]{sheavesingeometry}) in the special case where $i$ is the flat geometric embedding of \Cref{cameo}. In the next subections we will revisit their theorem, offering the widest generalisation we could find. 

\end{disc}

\subsection{Kan injectivity}

This subsection briefly recalls Kan injectivity in a $2$-category and revisits the theory of injective topoi (with respect to geometric embeddings). We will use this technology in the next subsection to provide a generalisation of \Cref{paradigm} to the $2$-category of topoi.

\begin{war}[The $2$-category of topoi]
Because the $2$-dimensional structure is so important, let us recall what we intend by $\mathsf{Topoi}$, the $2$-category of Grothendieck topoi. It has objects Grothendieck topoi, $1$-cells geometric morphisms (in the direction of the right adjoint) and $2$-cells natural transformations between the inverse (!) images. 
\end{war}

Kan injectivity was studied in the case of poset-enriched categories in \cite{escardo1997injective,escardo1998properly,escardo1999semantic,adamek2015kan} and recently generalised to $2$-categories in \cite{dlsolo}, we refer to the latter for the reader that is new to the concept. Below we recall an operative definition of Kan injectivity, specialised directly to the case of the $2$-category of topoi.

\begin{defn}[Kan injectivity, {\cite{dlsolo}}]
Let $f: \mathcal{F} \to \mathcal{G}$. We say that $\mathcal{E}$ is left (right) Kan injective with respect to $f$ if the functor \[(-) \circ f: \mathsf{Topoi}(\mathcal{G}, \mathcal{E}) \to \mathsf{Topoi}(\mathcal{F}, \mathcal{E})\] has a left (right) adjoint whose unit (counit) is an isomorphism.
\end{defn}

\begin{war}[$f^\sharp \dashv f_\sharp$]
It is useful to have a more compact notation for the functor $(-) \circ f: \mathsf{Topoi}(\mathcal{G}, \mathcal{E}) \to \mathsf{Topoi}(\mathcal{F}, \mathcal{E})$.  We shall call $f^\sharp$ such functor and $f_\sharp$ its right adjoint. Of course, other notations would be more traditional, but they would clash with the usual symbol for the inverse image functor.
\end{war}

\begin{war}[On right Kan extension in the $2$-category of topoi] \label{rightKanextension}
\[\begin{tikzcd}
	{\mathcal{F}} & {\mathcal{E}} \\
	{\mathcal{G}}
	\arrow["f"', from=1-1, to=2-1]
	\arrow["x", from=1-1, to=1-2]
	\arrow["{\mathsf{ran}_fx}"', dashed, from=2-1, to=1-2]
\end{tikzcd}\]
Recall from \cite{dlsolo} that $f_\sharp(x)$ is the right Kan extension $\mathsf{ran}_fx$ in the $2$-category of topoi (as opposed to the $2$-category of categories) and thus has the appropriate universal property in that $2$-category. We will see below that in some concrete instances (\Cref{h_*}), some counter-intuitive equations hold,  \[({\mathsf{ran}^{\mathsf{Topoi}}_fx})_* \cong \mathsf{lan}^{\mathsf{Cat}}_{f_*}x_*.\]
We do not know whether such an equation has to hold in general, but it is a great exemplification that one should be careful when manipulating these Kan extensions. 
\end{war}

\begin{defn}[Morphisms between injectives, {\cite{dlsolo}}]
A geometric morphism $g: \mathcal{E} \to \mathcal{D}$ between left (right) injective objects with respect to a morphism $f$ is a morphism of left (right) injectives  if it preserves the left (right) Kan extension as below. In {\cite{dlsolo}} these are called left (right) injective morphisms.

\[\begin{tikzcd}
	{\mathcal{F}} & {\mathcal{E}} & {\mathcal{D}} \\
	{\mathcal{G}}
	\arrow["f"', from=1-1, to=2-1]
	\arrow["x", from=1-1, to=1-2]
	\arrow["g", from=1-2, to=1-3]
	\arrow["{f_\sharp(x)}"{description}, from=2-1, to=1-2]
	\arrow["{f_{\sharp}(gx)}"', curve={height=12pt}, from=2-1, to=1-3]
\end{tikzcd}\]
\end{defn}

\begin{defn}[Relevant classes of flat morphisms: $\mathsf{Emb}_\beta \subset \mathsf{Emb}_\flat$] \label{classes}
We define two classes of geometric morphisms, one trivially contained in the other $\mathsf{Emb}_\beta \subset \mathsf{Emb}_\flat$. These two classes will be the main object of study of this paper.

\begin{itemize}
    \item[$(\mathsf{Emb}_\flat)$] contains flat embeddings,
    \item[$(\mathsf{Emb}_\beta)$] contains the flat embeddings of the form $\mathsf{Set}^X \to \mathsf{Sh}(\beta(X))$, where $X$ is a set and $\beta(X)$ is its space of ultrafilters. This is a special case of \Cref{cameo} where $B$ is the powerset $\mathcal{P}(X)$.
\end{itemize}
\end{defn}

\begin{notat}[$\beta$-complete and $\flat$-complete topoi] \label{complete}
Following \cite{dlsolo}, we call $\mathsf{RInj}(\mathsf{Emb}_\square)$ the  subcategories of topoi that are right Kan injective with respect to $\mathsf{Emb}_\square$, together with their morphisms of right injectives and all natural transformations. We shall call $\square$-complete a topos in $\mathsf{RInj}(\mathsf{Emb}_\square)$. We obtain a forgetful functor \[ \mathbb{U}_\beta^\flat: \mathsf{RInj}(\mathsf{Emb}_\flat) \to \mathsf{RInj}(\mathsf{Emb}_\beta).\]
\end{notat}

We now revisit a classical result from the theory of topoi, from the point of view of right Kan injectivity. 

\begin{prop}
Let $C$ be a category with finite limits. Then $\mathsf{Psh}(C)$ is right Kan injective with respect to embeddings of topoi.
\end{prop}
\begin{proof}
This is shown in \cite[Prop. 1.2]{johnstone1981injective}. It will be later useful though to provide explicitly the construction of the right adjoint and to inspect its properties. In order to do so, consider the diagram below where of course $i$ is an embedding and $x$ is any geometric morphism. 

\[\begin{tikzcd}
	{\mathcal{F}} & {\mathsf{Psh}(C)} \\
	{\mathcal{G}}
	\arrow["i"', hook, from=1-1, to=2-1]
	\arrow["x", from=1-1, to=1-2]
	\arrow["h"', dashed, from=2-1, to=1-2]
\end{tikzcd}\]

We shall describe the construction for $h$. Call $y: C \to \mathsf{Psh}(C)$ the Yoneda embedding. We define, \[h^* = \mathsf{lan}_{y}(i_*x^*y).\] It is clear that $h^*$ is cocontinuous and lex, because $i_*x^*y$ is a lex functor and it is easy to check that the geometric morphism induced by $h^*$ has the desired properties.
\end{proof}

\begin{rem}[A description for $h_*$] \label{h_*}
We note \emph{en passant} that the definition of $h^*$ provides a concrete description of its right adjoint by the following chain of isomorphisms,
\begin{align*}
    h_* \cong & \mathsf{lan}_{h^*}(1) \\
        \cong & \mathsf{lan}_{h^*}(\mathsf{lan}_y y) \\
        \cong & \mathsf{lan}_{h^*y}(y) \\
        \cong & \mathsf{lan}_{(\mathsf{lan}_{y}(i_*x^*y)y)}(y) \\
        \cong & \mathsf{lan}_{i_*x^*y}(y) \\
         \cong & \mathsf{lan}_{i_*x^*}(\mathsf{lan}_{y}(y)) \\
              \cong & \mathsf{lan}_{i_*}(x_*) \\
\end{align*}
\end{rem}

\subsection{Coherent topoi are Kan injective with regard to flat embeddings}

In this subsection we deliver our generalisation of \Cref{paradigm}. We shall start by an approximation of such result in the special case of strong flatness. The general case will follow a different proof strategy; yet, the proof below has some very interesting elements that, after this paper first appeared, have inspired the theory developed in \cite{dilibertiLogicConcepts2category2025,dilibertiCraigInterpolationSubgeometric2026,dilibertiConceptualCompletenessSubgeometric2026,libertiLogicsVsDoctrines2026}.

\begin{thm} \label{mainth1}
Coherent topoi are right Kan injective with respect to strongly flat embeddings.
\end{thm}

\begin{proof}
We start by saying that a coherent topos $\mathcal{E}$ admits a geometric embedding into a presheaf topos over a lex category $j: \mathcal{E} \hookrightarrow \mathsf{Psh}(C)$. Moeover, $C$ can be chosen to be the full subcategory of coherent objects, so that the direct image of the geometric embedding preserve directed colimits \cite[III.1.1(2)]{moerdijk2000proper}. In the diagram below, let $i$ be a strongly flat embedding of topoi and $x$ be any geometric morphism.

\[\begin{tikzcd}
	{\mathcal{L}} & {\mathcal{E}} \\
	{\mathcal{L'}} & {\mathsf{Psh}(C)}
	\arrow["x", from=1-1, to=1-2]
	\arrow["i"', hook, from=1-1, to=2-1]
	\arrow["j", hook, from=1-2, to=2-2]
	\arrow["h"', dashed, from=2-1, to=2-2]
\end{tikzcd}\]

Because $\mathsf{Psh}(C)$ is right Kan injective, $h$ exists. By \Cref{h_*}, we have a formula to compute its right adjoint which we recall, \[h_*  \cong \mathsf{lan}_{i_*}(j_*x_*).\]
To construct $i^\mathcal{E}_\sharp(x)$ we show that $j_*j^*h_* \cong h_*,$ this implies that $h$ factors through $\mathcal{E}$ because it means every object in the image of $h_*$ is automatically a $j$-sheaf. 
\begin{align*}
    j_*j^*h_*  \cong & j_*j^*\mathsf{lan}_{i_*}(j_*x_*) \\
                \cong & j_*\mathsf{lan}_{i_*}(j^*j_*x_*) \\
                  \cong & j_*\mathsf{lan}_{i_*}(x_*) \\
       (*)           \cong & \mathsf{lan}_{i_*}(j_*x_*) \\
                 \cong & h_*. \\
\end{align*}
We should justify the equation $(*)$, i.e. why $j_*$ preserves the Kan extension $\mathsf{lan}_{i_*}(x_*)$. To do so, we write down the formula to compute the Kan extension explicitly. Let us do that,

\[j_*\mathsf{lan}_{i_*}(x_*)(y) \cong j_*(\underset{i_*(d) \to y}{\mathsf{colim}} x_*(d))\] 

Now, because $i_*$ preserves finite colimits, the diagram indexing the colimit is filtered, and thus is preserved by $j_*$. But this is the same as saying that $j_*$ preserves the Kan extension.

\[j_*\mathsf{lan}_{i_*}(x_*)(y) \cong j_*(\underset{i_*(d) \to y}{\mathsf{colim}} x_*(d)) \cong   \underset{i_*(d) \to y}{\mathsf{colim}} j_*x_*(d)  \cong \mathsf{lan}_{i_*}(j_*x_*)(y).\] 

The discussion above shows that we can indeed define a functor $i^\mathcal{E}_{\sharp}$ in the diagram below with the property that $j_\mathcal{L'}^{\bullet}i^\mathcal{E}_{\sharp} \cong i^{\mathsf{Psh}(C)}_{\sharp}j^\bullet_{\mathcal{L}}$, which is a convoluted way to say that $i^{\mathsf{Psh}(C)}_{\sharp}j^\bullet_{\mathcal{L}}$ actually lands in $\mathsf{Topoi}(\mathcal{L'}, \mathcal{E})$. 

\[\begin{tikzcd}
	{\mathsf{Topoi}(\mathcal{L}, \mathcal{E})} && {\mathsf{Topoi}(\mathcal{L}, \mathsf{Psh}(C))} \\
	\\
	{\mathsf{Topoi}(\mathcal{L}', \mathcal{E})} && {\mathsf{Topoi}(\mathcal{L}', \mathsf{Psh}(C))}
	\arrow[""{name=0, anchor=center, inner sep=0}, "{i^{\mathsf{Psh}(C)}_\sharp}", curve={height=-12pt}, from=1-3, to=3-3]
	\arrow[""{name=1, anchor=center, inner sep=0}, "{i_{\mathsf{Psh}(C)}^\sharp}", curve={height=-12pt}, from=3-3, to=1-3]
	\arrow["{i_{\mathcal{E}}^\sharp}", curve={height=-12pt}, from=3-1, to=1-1]
	\arrow["{j_\mathcal{L}^\bullet}", from=1-1, to=1-3]
	\arrow["{j_\mathcal{L'}^\bullet}", from=3-1, to=3-3]
	\arrow["{i^\mathcal{E}_\sharp}", curve={height=-18pt}, dashed, from=1-1, to=3-1]
	\arrow["\dashv"{anchor=center}, draw=none, from=1, to=0]
\end{tikzcd}\]

We still have to show that $i^\sharp_\mathcal{E} \dashv i_\sharp^\mathcal{E}$. To do so, we will use the fact that both the functors $j^\bullet_{\square}$ are fully faithful (because $j$ is in first place). Indeed, it is enough to follow the following chain of equivalences to finish the proof. (We simplified the notation for the sake of readability).

\begin{align*}
    (\mathcal{L},\mathcal{E})(i^\sharp_\mathcal{E}-,-) \simeq & (\mathcal{L},\mathsf{Psh}(C))(j^\bullet_\mathcal{L}i^\sharp_\mathcal{E}-,j^\bullet_\mathcal{L}-) \\ 
    \simeq & (\mathcal{L},\mathsf{Psh}(C))(i^\sharp_{\mathsf{Psh}(C)}j^\bullet_\mathcal{L'}-,j^\bullet_\mathcal{L}-) \\
    \simeq & (\mathcal{L}',\mathsf{Psh}(C))(j^\bullet_\mathcal{L'}-,i_\sharp^{\mathsf{Psh}(C)}j^\bullet_\mathcal{L}-) \\
    \simeq & (\mathcal{L}',\mathsf{Psh}(C))(j^\bullet_\mathcal{L'}-,j_\mathcal{L'}^{\bullet}i^\mathcal{E}_{\sharp}-) \\
  \simeq &  (\mathcal{L}',\mathcal{E})(-,i^\mathcal{E}_{\sharp}-). \\
\end{align*}
\end{proof}

\begin{rem}[A comment on the proof] \label{morgan}
Another way to state the theorem above is to say that flat embeddings of topoi are  \textit{orthogonal on the left} to relatively tidy embeddings of topoi into presheaf topoi over lex categories. This may seem as a convoluted way of stating the theorem, but opens the door to a new theory of orthogonality. It would be interesting for example to know whether in the proof above $j$ can land into any Kan injective with respect to embeddings. On a similar note, it would be interesting to see how this theorem changes if we replace the notion of flatness, for example with that of purity (see \cite[III.5]{moerdijk2000proper}).
\end{rem}

\begin{rem}[Another comment on the proof]
One can look at our proof as a very anatomic analysis of \cite[IX.11.1]{sheavesingeometry}. Besides being more general, our proof has the perk of highlighting the key moments where the hypotheses are used, in such a way that the proof is parametric in those assumptions. For example, if $i_*$ was only preserving coproducts, a similar proof would work for $j_*$ preserving sifted colimits.
\end{rem}

We shall now see the proof of the general case of the theorem for flat embeddings. The first proof of this result originally appeared in a subsequent paper \cite[4.4.3]{dilibertiLogicConcepts2category2025}, in a much more general context and framework. Below, we reproduce an adaptation of that argument to this special case for the sake of completeness.

\begin{thm} \label{mainth}
Coherent topoi are right Kan injective with respect to flat embeddings.
\end{thm}
\begin{proof}
Let $i\colon\mathcal L\hookrightarrow\mathcal L'$ be a flat embedding
and $x\colon\mathcal L\to\mathcal E$ a geometric morphism, with
$\mathcal E$ coherent. We construct the right Kan extension
\[
\begin{tikzcd}[column sep=large, row sep=large]
\mathcal L
  \arrow[r, "x"]
  \arrow[d, hook, "i"']
  & \mathcal E \\
\mathcal L'
  \arrow[ur, dashed, "i_\sharp(x)"']
  &
\end{tikzcd}
\]
Let $C$ be the pretopos of coherent objects of
$\mathcal E$, and let $y\colon C\hookrightarrow\mathcal E$ be the
canonical embedding. Since $i_*$ preserves finite limits, finite coproducts and regular epimorphisms, the composite $i_*x^*y\colon C\to\mathcal L'$ is a morphism
of pretoposes. It therefore extends to a cocontinuous left exact functor
\[
  (i_\sharp(x))^*
  :=\mathsf{lan}_{y}(i_*x^*y)\colon\mathcal E\longrightarrow\mathcal L',
\]
which defines a geometric morphism
$i_\sharp(x)\colon\mathcal L'\to\mathcal E$.

For every $t\colon\mathcal L'\to\mathcal E$, restriction along $y$
and the adjunction $i^*\dashv i_*$ give natural bijections
(omitting $\mathsf{Topoi}$ from the hom-category notation):
\begin{align*}
  (\mathcal L',\mathcal E)(t,i_\sharp(x))
    \cong & \mathsf{Cat}(C,\mathcal L')
      (t^*y,i_*x^*y)\\
    \cong & \mathsf{Cat}(C,\mathcal L)
      (i^*t^*y,x^*y)\\
    \cong & (\mathcal L,\mathcal E)(i^\sharp(t),x).
\end{align*}

Thus $i^\sharp\dashv i_\sharp$, where $i^\sharp=(-)\circ i$.
The counit $i_\sharp(x)i\Rightarrow x$ restricts along $y$ to
\[
  i^*(i_\sharp(x))^*y
  \cong i^*i_*x^*y
  \xrightarrow{\ \sim\ }x^*y,
\]
which is invertible because $i$ is an embedding. By density of $y$,
the counit is invertible everywhere, proving right Kan injectivity.
\end{proof}

\begin{rem}
Our result improves \cite[IX.11.1]{sheavesingeometry} in two directions:
\begin{itemize}
    \item we showed that coherent topoi are injective with respect to all flat embeddings, as opposed to a very specific kind.
    \item we showed that coherent topoi are Kan (!) injective, which of course is a sharper property than being simply injective.
\end{itemize}
\end{rem}

\begin{rem}
Another version of \Cref{mainth} appears in \cite[Sec. 2]{johnstone1981gleason}. After showing that coherent locales are injective with respect to flat embeddings, Johsntone observes (right after Cor.\ 2.5) that his proofs are constructive and thus extend to locales internal to any topos. It would be interesting to use the fact that every coherent topos can be presented by an internal coherent locale in (a discrete power of) the classifier of the theory of objects, and compare the two theorems. Unfortunately this work does not come without technical effort, and would require a neat comparison between the localic and the topoc theoretic notion of flatness. 
\end{rem}

\begin{rem}[Characterise injectives with respect to flat embeddings?] \label{characterisation}

As we mentioned above, it is observed in \cite{carvalho2017kan} that right Kan injectives with respect to flat embeddings of locales are precisely coreflections of coherent locales. This leads to the simple conjecture that such a result could be true also for coherent topoi, which would set a perfect matching between coherent topoi and flat embeddings. In the case of locales it is easy to derive that every right Kan injective with respect to flat embeddings is a coreflection of a coherent locale because every locale flatly embeds in a coherent one (\Cref{exampleindcompl}). This does not seem to be as easily true though for the theory of topoi for two reasons. \begin{enumerate}
    \item  Because every topos is localic with respect to the classifier of the theory of objects $\mathsf{Set}[\mathbb{O}]$, we could try and simulate the proof of \Cref{exampleindcompl} from \cite[IX.10.3]{sheavesingeometry} for the case of locales internal to $\mathsf{Set}[\mathbb{O}]$, which would deliver the proof that every topos flatly embeds in a coherent one. Unfortunately the proof seem to rely on the use of choice, which fails in  $\mathsf{Set}[\mathbb{O}]$.
    \item For a topos $\mathcal{E}$, there is indeed an analog of the ideal-construction which is the $\mathsf{Ind}$-completion of a category, \[i: \mathcal{E} \to \mathsf{Ind}(\mathcal{E}).\]
     $i$ has both conceptually and technically the same behavior of the ideal-construction. Indeed $i$ preserves finite colimits, and one could show (with some work on size issues) that it has a lex left adjoint. Unfortunately though, $\mathsf{Ind}(\mathcal{E})$ is seldom a topos\footnote{Or even an infinitary-pretopos.} (see \cite{borceux1999left})!  One should say though that in the spirit of \cite{borceux1999left} it is believable that every \textit{noetherian} topos flatly embeds in a coherent one, and this observation would deserve further investigation. Could this be a characterization of noetherian topoi?
\end{enumerate} 
All in all, with the technology that is available today, we cannot characterise the topoi that are right Kan injective with respect to flat embeddings. Notice that this would be an extremely interesting task, as all these topoi must have enough points. The best result we can deliver follows.
\end{rem}

\begin{prop} \label{char2} Let $\mathcal{E}$ be a $\flat$-complete topos. Then $\mathcal{E}$ admits a geometric surjection from a coherent spatial topos. (In particular it has enough points.)

\end{prop}
\begin{proof}
Let $\mathcal{L}$ be the Barr cover of $\mathcal{E}$ (\cite[Thm. 4.1]{johnstone1985general}), and consider the diagram below, where $i: \mathcal{L} \to \mathcal{B}$ is constructed as in \Cref{cameo}.

\[\begin{tikzcd}
	{\mathcal{L}} & {\mathcal{E}} \\
	{\mathcal{B}}
	\arrow["q", two heads, from=1-1, to=1-2]
	\arrow["i"', hook, from=1-1, to=2-1]
	\arrow["h"', dashed, two heads, from=2-1, to=1-2]
\end{tikzcd}\]

$\mathcal{B}$ is coherent by design. Because $\mathcal{E}$ is right Kan injective with respect to flat embeddings, $h$ exists and because $q$ is a surjection, so must be $h$.
\end{proof}


\begin{rem}[A first encounter with Marmolejo]
In \cite{marmolejo1995ultraproducts} introduces the highly nontrivial framework of \textit{Łoś categories} to organise the fact that coherent topoi are $\beta$-complete. His point of view is different from ours, even though it really has the same motivations, which is to better understand the notion of ultracategory. 
\end{rem}

\subsection{Coherent Geometric Morphisms}

In this subsection we recall the definition of coherent geometric morphism between coherent topoi. The main result of the subsection is that coherent geometric morphisms are morphisms of right Kan injectives with respect to $\mathsf{Emb}_\flat$.

\begin{defn}[Coherent geometric morphism] \label{defnspartan}
A geometric morphism $f: \mathcal{E} \to \mathcal{D}$ between coherent topoi is coherent if its inverse image functors preserves coherent cobjects. We define $\mathsf{CohTopoi}$ to be the sub $2$-category of $\mathsf{Topoi}$ containing coherent topoi and coherent geometric morphisms.
\end{defn}

\begin{rem}
Coherent geometric morphism are a classic notion in the theory of coherent topoi, one can see for example \cite[9.2.8 and 9.2.9]{MR77}, but they do feature in many typical sources. \cite[Ch.~V, Example~3.2]{moerdijk2000proper} observes that they are all proper geometric morphisms. Also, for $d: \mathcal{C} \to \mathcal{D}$ a coherent morphism between (small) coherent categories, the associated geometric morphism between coherent topoi $\mathsf{Sh}(d): \mathsf{Sh}(\mathcal{D}) \to \mathsf{Sh}(\mathcal{C})$ is a coherent geometric morphism \cite[Vol. II, Exp.~VI, Cor.~3.3]{SGA4}.
\end{rem}

\begin{prop} \label{moderatesarekan}
Coherent geometric morphisms $f: \mathcal{E} \to \mathcal{D}$ between coherent topoi are morphisms of right Kan injective with respect to flat geometric embeddings.
\end{prop}
\begin{proof}

Let $i\colon\mathcal L\hookrightarrow\mathcal L'$ be a flat
geometric embedding and $x\colon\mathcal L\to\mathcal E$ a
geometric morphism. By \Cref{mainth}, the right Kan extensions
$i_\sharp(x)$ and $i_\sharp(fx)$ exist. We show that the canonical
comparison
\[
  f i_\sharp(x)\Rightarrow i_\sharp(fx)
\]
induced by the counit $i_\sharp(x)i\Rightarrow x$ is invertible.

Let $d$ be a coherent object of $\mathcal D$. Since $f$ is coherent,
$f^*(d)$ is coherent in $\mathcal E$. The construction in the proof
of \Cref{mainth} therefore gives natural isomorphisms
\[
\begin{aligned}
  \bigl(f i_\sharp(x)\bigr)^*(d)
    &= (i_\sharp(x))^*f^*(d)\\
    &\cong i_*x^*f^*(d)\\
    &\cong (i_\sharp(fx))^*(d).
\end{aligned}
\]
Under these identifications, the comparison at $d$ is the identity:
its transpose under $i^*\dashv i_*$ is the counit
$i^*i_*x^*f^*(d)\to x^*f^*(d)$. Thus the comparison is invertible
on coherent objects. Since coherent objects are dense in $\mathcal D$
and both inverse-image functors preserve colimits, it is invertible
everywhere. Hence $f$ preserves right Kan extensions along flat
geometric embeddings.

\end{proof}

We organise the results of this short subsection in the diagram below.

\[\begin{tikzcd}[ampersand replacement=\&]
	{\mathsf{RInj}(\mathsf{Emb}_\flat)} \&\& {\mathsf{RInj}(\mathsf{Emb}_\beta)} \\
	\& {\mathsf{CohTopoi}} \\
	\& {\mathsf{Topoi}}
	\arrow["{\mathbb{U}_\beta^\flat}"{description}, from=1-1, to=1-3]
	\arrow[curve={height=12pt}, from=1-1, to=3-2]
	\arrow[curve={height=-12pt}, from=1-3, to=3-2]
	\arrow[dashed, hook', from=2-2, to=1-1]
	\arrow[dashed, hook, from=2-2, to=1-3]
	\arrow[from=2-2, to=3-2]
\end{tikzcd}\]

\section{The emergence of ultrastructures à la Makkai} \label{sec2}
Historically, there were three approaches to ultracategories and ultrastructures in the literature \cite{makkai1987stone,marmolejo1995ultraproducts,lurieultracategories}. The aim of this section is to re-frame their wisdom in a cleaner setup and describe a new approach to the topic. Besides the technical choices of the authors, an ultrastructure on an category is meant to acknowledge the fact that, given a coherent theory $\mathbb{T}$, a set $X$, and an ultrafilter $U \in \beta(X)$, one can define the functor \[\int_X (-)dU: \mathsf{Mod}(\mathbb{T})^X \to \mathsf{Mod}(\mathbb{T}).\] 

Such functor takes an $X$-indexed family of models and computes the ultraproduct of those along the ultrafilter $U$. Our first task will be to reconstruct this functor from a more conceptual point of view, and see how it fits together with the theory of coherent topoi and injectivity with respect to flat embeddings.

\begin{constr}[The emergence of ultrastructures for $\beta$-complete topoi] \label{emergentultrastructures}
Consider now a map in $\mathsf{Emb}_\beta$ from \Cref{classes}. We shall recall some notation for the sake of readability. Let $X$ be a set and $i_X: X \to \beta(X)$ be the inclusion of $X$ (seen as a discrete space) in its space of ultrafilters. At the level of locales, this correspond to the morphism $i_{\mathcal{P}X} : \mathcal{P}X \to \mathsf{Idl}(\mathcal{P}(X))$ from \Cref{exampleindcompl} (where $\mathcal{P}X$ the powerset of $X$). Then, the corresponding flat geometric morphism, $\mathsf{Sh}(\mathcal{P}X) \to \mathsf{Sh}(\mathsf{Idl}(\mathcal{P}X))$ is nothing but the geometric morphism, \[i_X: \mathsf{Set}^X \to \mathsf{Sh}(\beta(X)).\] Now let $\mathcal{E}$ be a $\beta$-complete topos and consider the diagram below. Of course, recall that coherent topoi are indeed $\beta$-complete as we showed in the previous section.

\[\begin{tikzcd}
	{\mathsf{Set}^X} & {\mathcal{E}} \\
	{\mathsf{Sh}(\beta(X))}
	\arrow["f", from=1-1, to=1-2]
	\arrow["i_X"', from=1-1, to=2-1]
	\arrow["{i_\sharp(f)}"', curve={height=12pt}, dashed, from=2-1, to=1-2]
\end{tikzcd}\]

We claim that the family of maps $i^X_\sharp$, with $X$ a set, essentially account for the construction of ultrapowers. While our approach is original, this observation is not entirely new to the literature, and indeed Marmolejo's PhD thesis \cite{marmolejo1995ultraproducts} is built on this intuition. Before we give any definition of ultracategory or ultrastructure, we shall clarify the intution we just suggested, i.e. we shall relate, at least informally, $i^X_\sharp$ to the construction of ultrapowers. To do so, consider the diagram below (where we avoid all the $X$-dependencies as they play no role).

\[\begin{tikzcd}
	& X && {\mathsf{Set}^X} & {\mathcal{E}} \\
	{(\beta(X),\mathsf{disc})} & {(\beta(X),\mathsf{Stone})} & {\mathsf{Set}^{\beta(X)}} & {\mathsf{Sh}(\beta(X))}
	\arrow["f", from=1-4, to=1-5]
	\arrow["i", from=1-4, to=2-4]
	\arrow["{i_\sharp(f)}"', curve={height=12pt}, dashed, from=2-4, to=1-5]
	\arrow["q"', two heads, from=2-1, to=2-2]
	\arrow["j"', from=1-2, to=2-1]
	\arrow["i", from=1-2, to=2-2]
	\arrow["j"', from=1-4, to=2-3]
	\arrow["q"', from=2-3, to=2-4]
\end{tikzcd}\]

By $(\beta(X),\mathsf{Stone})$ we simply mean the usual topology on the space of ultrafilters, which corresponds to the ideal completion of the powerset of $X$. By inspecting this diagram, we can look at what happens at the level of hom-categories.

\[\begin{tikzcd}
	{\mathsf{pt}(\mathcal{E})^X} & {\mathsf{Topoi}(\mathsf{Sh}(\beta(X)), \mathcal{E})} & {\mathsf{Topoi}(\mathsf{Set}^{\beta(X)}, \mathcal{E})} \\
	{\mathsf{Topoi}(\mathsf{Set}, \mathcal{E})^X} & {\mathsf{Topoi}(\mathsf{Set}^X, \mathcal{E})} & {\mathsf{pt}(\mathcal{E})^{\beta(X)}}
	\arrow["\simeq"', no head, from=2-1, to=2-2]
	\arrow["\simeq", no head, from=2-1, to=1-1]
	\arrow["{q^\sharp}", from=1-2, to=1-3]
	\arrow["\simeq"', no head, from=1-3, to=2-3]
	\arrow["{i_{\sharp}}", from=2-2, to=1-2]
\end{tikzcd}\]

Altogether, and with a bit of abuse of notation that ignores the equivalence of categories, we obtain a functor

\begin{equation} \label{ultrastructure}
    q^\sharp_X i_\sharp^X: \mathsf{pt}(\mathcal{E})^X \to \mathsf{pt}(\mathcal{E})^{\beta(X)}.
\end{equation}

If we now transpose this functor, we obtain the pairing below, which we shall denote suggestively by an integral notation,

\begin{equation} \label{ultrastructureoriginal}
    \int_X(-)d(-): \mathsf{pt}(\mathcal{E})^X \times \beta(X) \to \mathsf{pt}(\mathcal{E}).
\end{equation}

Now, to see that this construction matches the intuition we have provided, we choose a point of the topos $\mathsf{Sh}(\beta(X))$, which is the same as an ultrafilter $U \in \beta(X)$. We shall denote such a point $U: \mathsf{Set} \to \mathsf{Sh}(\beta(X))$. Then, it was observed by Marmolejo (and possibly by others before) that $(i_\sharp f) \circ U$ coincides with the ultraproduct $\frac{\Pi f}{U}$ of $f$ along $U$ (see for example the introduction of \cite{marmolejo2005locale} or the more extensive the discussion in \cite[Chap 2]{marmolejo1995ultraproducts}). Thus we obtain the chain of isomorphisms we wanted to show, \[\int f dU = (i_\sharp f) \circ U \cong  \frac{\Pi f}{U}.\]

The isomorphism above also clarifies the sense in which the ultrastructure is essentially accounted by $i^X_\sharp$, while $q^\sharp_X$ is evaluating the parametric ultraproduct along a specific ultrafilter.
\end{constr}

\begin{disc}
The general program of ultrastructures and ultracategories is to encapsulate the pairing above in an axiomatic way. 
Thus an ultrastructure on a category $\mathcal{A}$ should be a(n indexed family of) pairing(s),

\[\int_X(-)d(-): \mathcal{A}^X \times \beta(X) \to \mathcal{A}\]

that has all the perks and features of that introduced in \Cref{ultrastructureoriginal}. Taking inspiration from \Cref{emergentultrastructures} and particularly from \Cref{ultrastructure}, an ultrastructure on $\mathcal{A}$ may be a family of (pseudo)sections $\Sigma_X$ of the functors $i^\sharp_X: \mathcal{A}^{\beta{X}} \to \mathcal{A}^X$ for each set $X$, \[\Sigma_X: \mathcal{A}^X \to \mathcal{A}^{\beta{X}}.\]
\end{disc}

\begin{constr}[A $2$-dimensional aspect of ultrastructures] \label{2dimensional}
The problem of this approach is that, as we discussed above, the universal property of $\mathcal{E}$ with respect to the construction of ultraproducts is accounted for by the functor $i^X_\sharp$, and the composition in \Cref{ultrastructure} loses information by composing with $q^\sharp_X$. Let us be more precise. Recall that $i_\sharp(x)$ coincides with a right Kan extension, \[i^X_\sharp(x) \cong \mathsf{ran}_{i_X}x,\] and thus there is a $2$-dimensional aspect to its universal property in $\mathsf{Topoi}$. Our definition of ultrastructure must witness at least a trace of such universal property. Consider the example below,

\[\begin{tikzcd}
	&& {\mathsf{Set}^Y} && {\mathcal{E}} \\
	\\
	{\mathsf{Set}^X} && {\mathsf{Sh}(\beta(Y))} \\
	\\
	{\mathsf{Sh}(\beta(X))}
	\arrow["f", from=1-3, to=1-5]
	\arrow["{i^Y}"', from=1-3, to=3-3]
	\arrow["{i^X}"', from=3-1, to=5-1]
	\arrow["g", from=3-1, to=3-3]
	\arrow["{i^Y_\sharp(f)}"', from=3-3, to=1-5]
	\arrow["{i^X_\sharp(g)}"', from=5-1, to=3-3]
	\arrow[""{name=0, anchor=center, inner sep=0}, "{i^X_\sharp(i^Y_\sharp(f)g)}"', shift right=2, curve={height=45pt}, from=5-1, to=1-5]
	\arrow["{\Delta_{fg}^{XY}}"{description}, shorten >=3pt, Rightarrow, from=3-3, to=0]
\end{tikzcd}\]

where $i_\sharp(g)$ exists because $\mathsf{Sh}(\beta(Y))$ is a coherent topos itself. Then, by the universal property of right Kan extensions, we obtain the natural transformation in the diagram,

\[(\Delta_{XY})_{fg} : \mathsf{ran}_{i^Y}f  \circ \mathsf{ran}_{i^X}g  \Rightarrow \mathsf{ran}_{i^X}(\mathsf{ran}_{i^Y}f \circ g).\]

Collecting the data of these natural transformations we obtain the diagram below.

\[\begin{tikzcd}
	{\beta Y^X \times \mathsf{pt}(\mathcal{E})^Y} & {\beta Y^{X} \times \mathsf{pt}(\mathcal{E})^{\beta Y}} & {\mathsf{pt}(\mathcal{E})^X} \\
	\\
	{\beta Y^{\beta(X)} \times \mathsf{pt}(\mathcal{E})^{\beta Y}} && {\mathsf{pt}(\mathcal{E})^{\beta X}}
	\arrow["{i_\sharp^X \times i_\sharp^Y}"{description}, from=1-1, to=3-1]
	\arrow["{i_\sharp^X}"{description}, from=1-3, to=3-3]
	\arrow["{1 \times i_\sharp^Y}", from=1-1, to=1-2]
	\arrow["{\circ_{\beta Y}}", from=1-2, to=1-3]
	\arrow["{\circ_{\beta Y}}"', from=3-1, to=3-3]
	\arrow["{\Delta_{XY}}"{description}, Rightarrow, from=3-1, to=1-3]
\end{tikzcd}\]


At this stage we have no guarantee that the data of $\Sigma_X$ and $\Delta_{XY}$ is enough to recover the whole structure $\mathsf{pt}(\mathcal{E})$ is canonically equipped with, but we do know that this structure has at least two important features:
\begin{itemize}
    \item $\Sigma_X$ imitates $q^\sharp_X i_\sharp^X$ from \Cref{ultrastructure} for the construction of ultrapowers of models, and this is even more transparent when we transpose it as in \Cref{ultrastructureoriginal}.
    \item $\Delta_{XY}$ is a trace of the universal properties of $i^X_\sharp$, which at least recovers the interaction between the indexing sets of the ultrastructure.
\end{itemize}

In the next subsection we will use the discussion above as a motivation for our definition of ultrastructure $(\Sigma_X, \Delta_{XY})$ on an accessible category $\mathcal{A}$ with directed colimits. Before moving to that, we shall make a couple of comments to compare our discussion with the literature.
\end{constr}

\begin{disc}[Marmolejo's Łoś categories] \label{loslos}
There are several ways to circumvent the probelm that the functor $q^\sharp_X i_\sharp^X$ in \Cref{emergentultrastructures} loses some information about the ultrastructure, and all these approaches are indeed related. Another way to keep track of the fact that a coherent topos is $\flat$-complete is due to Marmolejo \cite{marmolejo1995ultraproducts} and in its essence was (possibly independently) rediscovered by Lurie more recently \cite{lurieultracategories}. Let $\mathcal{E}$ be a topos and define, 
\begin{align*}
    \mathcal{E}^{\bullet} : & \mathsf{Sp}^\circ \to \mathsf{Cat} \\
    & X \mapsto \mathsf{Topoi}(\mathsf{Sh}(X),\mathcal{E})
\end{align*}
This derivator-like approach gives us a representation of topoi as $\mathsf{Sp}$-indexed categories and provides a $2$-functor, \[\mathsf{Topoi} \to \mathsf{Prestk}(\mathsf{Sp}).\]

It is clear that if $\mathcal{E}$ is a coherent topos, then for all ultrafinite maps $f$ (\Cref{generalities2}), $\mathcal{E}^\bullet(f)$ has a right adjoint (namely $f_\sharp$), and this could be chosen as a way to record the ultrastructure over $\mathcal{E}$. Marmolejo makes this choice in his PhD thesis \cite[Def 3.13]{marmolejo1995ultraproducts}, and defines a \textit{Łoś category}  to be precisely a $\mathsf{Sp}$-indexed category with the property above. We'll come back to Marmolejo's approach later and discuss the advantages and disadvantages of such choices. For the moment, just notice that a Łoś category is -- at least in principle -- a much less economical object than a triple $(\mathcal{A}, \Sigma_X, \Delta_{XY})$, for several reasons. For example, we have to specify a category for each topological space, whereas the triple $(\mathcal{A}, \Sigma_X, \Delta_{XY})$ specifies only one category.
\end{disc}

\begin{disc}[Lurie and Makkai]
Both Makkai and Lurie have an approach which is very much in the spirit of \Cref{emergentultrastructures} and \Cref{2dimensional} and to be more precise, Lurie's notion is essentially identical to the one we will propose later.

They both used this structure to show that every coherent topos can be recovered by its category of points plus the additional data of its ultrastructure. Notice that if we look at this result from the point of view of this paper, their result is very surprising for two reasons.

\begin{itemize}
    \item The maps of the form $\mathsf{Set}^X \to \mathsf{Sh}(\beta X)$ are, at least in principle, very few among flat geometric embeddings. Conceptual completeness seem to suggest that the $2$-functor \[\mathbb{U}_\beta^\flat: \mathsf{RInj}(\mathsf{Emb}_\flat) \to \mathsf{RInj}(\mathsf{Emb}_\beta)\] is a biequivalence of categories. Unfortunately we did not manage to show such a theorem.
    \item Differently from the localic case, there is no evidence that being $\beta$- or $\flat$-complete characterizes (co)retractions of coherent topoi.
\end{itemize}
\end{disc}

\section{Modeling Ultrastructures in Formal Model Theories} \label{sec3}

In this section we provide candidate definitions of ultracategory and discuss their relationships. As we mentioned several times since the opening of the paper, for us, an ultrastructure should imitate the construction of ultraproducts \[\int_X (-)dU: \mathsf{Mod}(\mathbb{T})^X \to \mathsf{Mod}(\mathbb{T})\] defined on the category of models of a coherent theory. This means in particular that an ultrastructure should be defined on something that looks like, behaves like and bites like a category of models of a geometric theory. As we have discussed in \cite{thlo}, there are several solutions available in the literature to  treat \textit{formally} categories of models of geometric theories. Among the options we can choose:

\begin{itemize}

    \item[$(\mathsf{Acc}_\omega)$] accessible categories with directed colimits, those offer the most classically developed approach to the topic, dating back to \cite{Makkaipare} and \cite{adamek_rosicky_1994}.
    \item[$(\mathsf{AEC})$] abstract elementary classes were popularized by Shelah \cite{shelah2009classification} in the model theory community as a general framework to have a syntax-independent treatment of categories of models of infinitary first order theories. This framework is a special case of that of accessible categories \cite{beke2012abstract}.
    \item[$(\mathsf{BIon})$] (bounded) ionads, this is a much newer framework, introduced by Garner \cite{ionads} with a geometric point of view, and studied from a more logical point of view in \cite{thlo}.
\end{itemize}

Older and/or less developed approaches can be seen through the lenses of formal model theory:

\begin{itemize}
    \item[($\mathsf{Cat}_{\mathsf{Sp}}$)] $\mathsf{Sp}$-indexed categories were used by  both Marmolejo and Lurie as an intermediate step to show properties of their ultracategories. This $2$-category deserves a separate study. At a phenomenological level we already have evidences that it is related to $\mathsf{Acc}_\omega$, see for example \cite[Prop. 6.14]{marmolejo1995ultraproducts}.
    \item[({$\mathsf{Grp}[\mathsf{Sp}]$})]  Topological groupoids and topological categories are classically understood to account for the Galois theory of a topos. Yet because a (large) topological groupoid can be seen as the groupoid of models of a topos equipped with the logical topology \cite{awodey2013first}, we see that those are perfect placeholders to perform formal model theory.
\end{itemize}

 In the discussion that follows we will instantiate definitions of ultracategory in some of those arenas of formal model theory: $\mathsf{Acc}_\omega, \mathsf{BIon}$ and $\mathsf{Acc}_\omega^{\mathsf{Sp}}$, which is a well-behaved variant of $\mathsf{Cat}_{\mathsf{Sp}}$ in the list above.

\subsection{Ultrastructures in $\mathsf{Acc}_\omega$} \label{sec:ultrastructures}

\begin{defn}[Ultrastructures and ultracategories in $\mathsf{Acc}_\omega$] \label{def:ultracategory}
Let $\mathcal{A}$ be an accessible category with directed colimits. An \textit{ultrastructure} on $\mathcal{A}$ is specified by the following data:

\begin{itemize}
    \item[$(\Sigma, \epsilon)$] \textit{functorial ultraproducts}: a family (pseudo)sections $\Sigma_X: \mathcal{A}^X \to \mathcal{A}^{\beta{X}}$ of the functor $i^\sharp_X: \mathcal{A}^{\beta{X}} \to \mathcal{A}^X$ for each set $X$, accompanied by a natural isomorphism $\epsilon_X : i^\sharp_X\Sigma_X \stackrel{\sim}{\to} 1$. We shall use the notation below for the transpose of $\Sigma_X$, \[\int_X(-)d(-): \mathcal{A}^X \times \beta(X) \to \mathcal{A}.\]
    \item[$(\Delta)$] \textit{algebraic universality}: a family of natural transformations $\Delta_{XY}$ filling the diagrams below. We will write $\Delta_{XY} : \Sigma_Y b \circ \Sigma_X a \Rightarrow \Sigma_X(\Sigma_Y b \circ a). $
\[\begin{tikzcd}
	{\beta Y^X \times \mathcal{A}^Y} & {\beta Y^{X} \times \mathcal{A}^{\beta Y}} & {\mathcal{A}^X} \\
	\\
	{\beta Y^{\beta(X)} \times \mathcal{A}^{\beta Y}} && {\mathcal{A}^{\beta X}}
	\arrow["{\Sigma_X \times \Sigma_Y}"{description}, from=1-1, to=3-1]
	\arrow["{\Sigma_X}"{description}, from=1-3, to=3-3]
	\arrow["{1 \times \Sigma_Y}", from=1-1, to=1-2]
	\arrow["{\circ_{\beta Y}}", from=1-2, to=1-3]
	\arrow["{\circ_{\beta Y}}"', from=3-1, to=3-3]
	\arrow["{\Delta_{XY}}"{description}, Rightarrow, from=3-1, to=1-3]
\end{tikzcd}\]
\end{itemize}

The data above is subject to some compatibility axioms listed below.

\begin{enumerate}
    \item the pastings in the diagram below both yield a isomorphism, where $\mathsf{Mono}(X,Y)$ the set of injections from $X$ to $Y$.
\[\begin{tikzcd}
	{\mathsf{Mono}(X,Y)\times\mathcal{A}^Y} & {\beta Y^X \times \mathcal{A}^Y} & {\beta Y^{X} \times \mathcal{A}^{\beta Y}} \\
	&& {\mathcal{A}^X} \\
	& {\beta Y^{\beta(X)} \times \mathcal{A}^{\beta Y}} & {\mathcal{A}^{\beta X}} & {\mathcal{A}^X}
	\arrow["{\Sigma_X \times \Sigma_Y}"{description}, from=1-2, to=3-2]
	\arrow["{\Sigma_X}"{description}, from=2-3, to=3-3]
	\arrow["{1 \times \Sigma_Y}", from=1-2, to=1-3]
	\arrow["{\circ_{\beta Y}}", from=1-3, to=2-3]
	\arrow["{\circ_{\beta Y}}"', from=3-2, to=3-3]
	\arrow["{i^\sharp}", from=3-3, to=3-4]
	\arrow["{\Delta_{XY}}"{description}, Rightarrow, from=3-2, to=1-3]
	\arrow["{i_Y \times 1}", from=1-1, to=1-2]
\end{tikzcd}\]

    \item \textit{skew associativity}: by applying algebraic universality, the two possible ways to reduce the composition on the left to the one on the right produce the same map.
    
    \adjustbox{scale=0.8,center}{
\begin{tikzcd}[ampersand replacement=\&,sep=3.25em]
    \& {\Sigma_Yc \circ  \Sigma_Z(\Sigma_X b \circ a)} \& { \Sigma_Z(\Sigma_Yc \circ \Sigma_X b \circ a)} \\
    {\Sigma_Yc \circ \Sigma_X b \circ \Sigma_Za} \&\&\& { \Sigma_Z(\Sigma_X(\Sigma_Yc \circ  b) \circ a)} \\
    \&\& { \Sigma_X(\Sigma_Yc \circ b) \circ \Sigma_Za}
    \arrow[Rightarrow, from=1-2, to=1-3]
    \arrow[curve={height=-12pt}, Rightarrow, from=1-3, to=2-4]
    \arrow[curve={height=-12pt}, Rightarrow, from=2-1, to=1-2]
    \arrow[curve={height=12pt}, Rightarrow, from=2-1, to=3-3]
    \arrow[curve={height=6pt}, Rightarrow, from=3-3, to=2-4]
\end{tikzcd}}

\end{enumerate}

 An \textit{ultracategory} in $\mathsf{Acc}_\omega$ is an object $\mathcal{A}$ equipped with an ultrastructure $(\Sigma^\mathcal{A}_\bullet, \Delta_{\bullet,\bullet})$. We will often blur the distinction between an ultracategory $(\mathcal{A}, \Sigma^\mathcal{A}_\bullet, \Delta_{\bullet,\bullet})$ and its underlying category $\mathcal{A}$ to simplify the notation. 
\end{defn}

\newcommand{\A}{\mathcal A}
\newcommand{\B}{\mathcal B}

Now let \(\mathcal A,\mathcal B\in\Acc_\omega\) be equipped with ultrastructures, denoted by
\((\Sigma^{\A}_\bullet,\epsilon^{\A}_\bullet,\Delta^{\A}_{\bullet,\bullet})\)
and \((\Sigma^{\B}_\bullet,\epsilon^{\B}_\bullet,\Delta^{\B}_{\bullet,\bullet})\),
respectively. 

\begin{defn}[Ultrafunctors in $\Acc_\omega$]  

An \emph{ultrafunctor} \((f,\delta_\bullet):\A\to \B\) consists of a
functor \(f:\A\to \B\) preserving directed colimits, together with an
invertible natural transformation filling the following square for
each set \(X\):
\[
\begin{tikzcd}[column sep=4em,row sep=2.5em]
A^X
  \arrow[r,"f^X"]
  \arrow[d,"\Sigma_X^A"']
&
B^X
  \arrow[d,"\Sigma_X^B"]
\\
A^{\beta X}
  \arrow[r,"f^{\beta X}"']
&
B^{\beta X}
\arrow[Rightarrow, from=2-1,to=1-2,"\delta_X"'{description}]
\end{tikzcd}
\]
These transformations are required to satisfy the following
compatibility conditions.
\begin{enumerate}[label=(\arabic*),leftmargin=2em]
\item For every set \(X\), the unit diagram below commutes, where we identify its upper-left vertex
with \(i_X^\sharp f^{\beta X}\Sigma_X^A\),
\[
\begin{tikzcd}[column sep=1.5em,row sep=2.5em]
f^X i_X^\sharp\Sigma_X^{\A}
  \arrow[rr,Rightarrow,"i_X^\sharp\delta_X"]
  \arrow[dr,Rightarrow,"f^X\epsilon_X^{\A}"']
&&
i_X^\sharp\Sigma_X^{\B} f^X
  \arrow[dl,Rightarrow,"\epsilon_X^{\B} f^X"]
\\
&
f^X
&
\end{tikzcd}
\]

\item For every \(a:X\to\beta Y\) and \(b:Y\to A\), the diagram below commutes,
\[
\begin{tikzcd}[column sep=4em,row sep=2.8em]
f\Sigma_Y^{\A} b\circ\Sigma_Xa
  \arrow[r,Rightarrow,"{\delta_Y(b)\,\Sigma_Xa}"]
  \arrow[d,Rightarrow,"{f\Delta^{\A}_{XY}(a,b)}"']
&
\Sigma_Y^{\B}(fb)\circ\Sigma_Xa
  \arrow[dd,Rightarrow,"{\Delta^{\B}_{XY}(a,fb)}"]
\\
f\Sigma_X^{\A}(\Sigma_Y^{\A} b\circ a)
  \arrow[d,Rightarrow,"{\delta_X(\Sigma_Y^{\A} b\circ a)}"']
&
\\
\Sigma_X^{\B}(f\Sigma_Y^{\A} b\circ a)
  \arrow[r,Rightarrow,"{\Sigma_X^{\B}(\delta_Y(b)\,a)}"']
&
\Sigma_X^{\B}(\Sigma_Y^{\B}(fb)\circ a)
\end{tikzcd}
\]
\end{enumerate}
\end{defn}

\begin{defn}[Transformations of ultrafunctors in $\Acc_\omega$]
Let \((f,\delta_\bullet),(g,\gamma_\bullet):\A\to \B\) be ultrafunctors.
A \emph{transformation of ultrafunctors}
\(\alpha:(f,\delta_\bullet)\Rightarrow(g,\gamma_\bullet)\)
is a natural transformation \(\alpha:f\Rightarrow g\) such that,
for every set \(X\), the following diagram commutes:
\[
\begin{tikzcd}[column sep=3.5em,row sep=2.5em]
f^{\beta X}\Sigma_X^{\A}
  \arrow[r,Rightarrow,"\delta_X"]
  \arrow[d,Rightarrow,"{\alpha^{\beta X}\Sigma_X^{\A}}"']
&
\Sigma_X^{\B} f^X
  \arrow[d,Rightarrow,"{\Sigma_X^{\B}\alpha^X}"]
\\
g^{\beta X}\Sigma_X^{\A}
  \arrow[r,Rightarrow,"\gamma_X"']
&
\Sigma_X^{\B} g^X .
\end{tikzcd}
\]
\end{defn}

\begin{defn}[The $2$-category $\mathsf{Ult}$ of Ultracategories in $\Acc_\omega$] \label{def:ult}
Ultracategories, ultrafunctors and transformations form a $2$-category $\mathsf{Ult}$ equipped with a locally faithful forgetful $2$-functor, 
\[\mathbb{U}: \mathsf{Ult} \to \mathsf{Acc}_\omega.\]
\end{defn}

\begin{prop} \label{th:ptliftult}
Let $\mathcal{E}$ be a topos in $\mathsf{RInj}(\mathsf{Emb}_\beta)$. Then its category of points $\mathsf{pt}(\mathcal{E})$ admits an ultrastructure and every morphism of right injectives induces an ultrafunctor. Thus, we obtain a lift of $2$-functor $\mathsf{pt}$ of points as below so that the diagram commutes,

\[\begin{tikzcd}
	{\mathsf{CohTopoi}} & {\mathsf{RInj(\mathsf{Emb}_\beta)}} & {\mathsf{Ult}} \\
	& {\mathsf{Topoi}} & {\mathsf{Acc}_\omega}
	\arrow[from=1-1, to=2-2]
	\arrow[from=1-3, to=2-3]
	\arrow["{\mathsf{pt}}"', from=2-2, to=2-3]
	\arrow[dashed, from=1-2, to=1-3]
	\arrow[from=1-1, to=1-2]
	\arrow[from=1-2, to=2-2]
\end{tikzcd}\]
\end{prop}
\begin{proof}
Most of the proof follows directly from the discussion in \Cref{emergentultrastructures} and \Cref{2dimensional}. We should show that for $\mathcal{E}$ be $\beta$-complete topos, the ultrastructure we discussed verify the axioms (1) and (2) above. We shall only discuss (2), because (1) follows from a very similar reasoning and is much less interesting from a conceptual point of view. Going back to \Cref{2dimensional}, checking the skew associativity leads us to compare the two natural transformations below,

\adjustbox{scale=0.8,center}{%
\begin{tikzcd}
	& {\mathsf{ran}_{i^Y}c  \circ \mathsf{ran}_{i^Z}(\mathsf{ran}_{i^X} b \circ a)} & { \mathsf{ran}_{i^Z}(\mathsf{ran}_{i^Y}c  \circ\mathsf{ran}_{i^X} b \circ a)} \\
	{\mathsf{ran}_{i^Y}c \circ \mathsf{ran}_{i^X} b \circ \mathsf{ran}_{i^Z}a} &&& { \mathsf{ran}_{i^Z}(\mathsf{ran}_{i^X}(\mathsf{ran}_{i^Y}c  \circ b) \circ a)} \\
	&& { \mathsf{ran}_{i^X} (\mathsf{ran}_{i^Y}c \circ b) \circ \mathsf{ran}_{i^Z}a}
	\arrow[curve={height=-12pt}, Rightarrow, from=2-1, to=1-2]
	\arrow[Rightarrow, from=1-2, to=1-3]
	\arrow[curve={height=-12pt}, Rightarrow, from=1-3, to=2-4]
	\arrow[curve={height=12pt}, Rightarrow, from=2-1, to=3-3]
	\arrow[curve={height=6pt}, Rightarrow, from=3-3, to=2-4]
\end{tikzcd}
}
And indeed the fact that twose two paths coincide correspond to the skew associativity structure of Kan extensions. The fact that morphisms of right Kan injectives induce ultrafunctors is completely evident.

\end{proof}

\begin{rem}[A comment on Lurie's ultracategories]
We should compare our notion to Lurie's one from \cite[Def. 1.3.1]{lurieultracategories}. Uncurrying our presentation, axiom (1) corresponds to axiom (B) in Lurie, while the skew associativity axiom (2) corresponds to axiom (C) in Lurie. Axiom (A) in Lurie can be depicted as follows, again where we identify the upper-left vertex with \(i_X^\sharp(\Sigma_Y b\circ\Sigma_X a)\), and the component of \(\epsilon_X\) is taken at \(\Sigma_Y b\circ a\):
\[
\begin{tikzcd}[column sep=20pt, row sep=20pt]
\Sigma_Y b\circ a
  \arrow[rr, Rightarrow, "i_X^\sharp\Delta_{XY}"]
  \arrow[dr, equal]
&&
i_X^\sharp\Sigma_X(\Sigma_Y b\circ a)
  \arrow[dl, Rightarrow, "\epsilon_X"]
\\
&
\Sigma_Y b\circ a.
&
\end{tikzcd}
\]
An ultracategory in $\Acc_\omega$ satisfying this additional condition is precisely an ultracategory in Lurie's sense whose underlying category is an accessible category with directed colimits by design. However, our notions of ultrafunctors and transformations coincide with the notions of the same name in \cite[Def.\ 1.4.1 and 1.4.2]{lurieultracategories}. Together with the fact that categories of models of coherent theories are accessible and closed under directed colimits, this implies that Makkai's strong conceptual completeness result recalled in the Introduction holds for our notion of ultracategories (see also \Cref{representation}).

We believe that our treatment offers a cleaner approach than Lurie's, which can be seen for example comparing our \Cref{emergentultrastructures} and especially \Cref{2dimensional} with his \cite[Exa. 1.3.8, which relies on 1.2.2, 1.2.8]{lurieultracategories} where we are essentially proving the same thing.
\end{rem}


\begin{rem}[Ultrafunctors and coherent objects] \label{rep1}
 Consider a coherent topos $\mathcal{E}$ and a coherent object $e: \mathcal{E} \to \mathsf{Set}[\mathbb{O}]$ encoded as a geometric morphism into the classifier of the theory of objects. It is easy to see that $e$ is a coherent geometric morphism between right Kan injectives with respect to $\mathsf{Emb}_\beta$, and thus the corresponding category of points gives us an ultrafunctor, $\mathsf{pt}(e) :\mathsf{pt}(\mathcal{E}) \to \mathsf{Set}$ where $\mathsf{Set}$ has \textit{the canonical ultrastructure}. It follows that we have a functor, \[\mathsf{Coh}(\mathcal{E}) \to \mathsf{Ult}(\mathsf{pt}(\mathcal{E}),\mathsf{Set}).\]
Notice though, that the assumption that $e$ is coherent is crucial in order to obtain a coherent geometric morphism, and thus we cannot guarantee any functor $\mathcal{E} \to \mathsf{Ult}(\mathsf{pt}(\mathcal{E}),\mathsf{Set}).$ Ultrafunctors cannot recover the whole topos, at least not on the spot. This behavior is related to the fact that the $2$-category $\mathsf{CohTopoi}$ is mostly related to the $2$-category of pretopoi, as opposed to the $2$-category of coherent topoi and geometric morphisms between them.
\end{rem}



\subsubsection{Left Ultrafunctors} \label{weak}

In the previous remark we found a representation of the pretopos of coherent objects via $\mathsf{Set}$-ultrafunctors, $\mathsf{Coh}(\mathcal{E}) \to \mathsf{Ult}(\mathsf{pt}(\mathcal{E}),\mathsf{Set})$. We shall discuss how to improve such embedding into a representation of the whole topos $\mathcal{E}$. Of course, we must pay the price of having a more flexible notion of ultrafunctor. The construction above was based on the fact that we can recover coherent objects by coherent geometric morphism,

\[\mathsf{Coh}(\mathcal{E}) \hookrightarrow \mathsf{CohTopoi}(\mathcal{E},\mathsf{Set}[\mathbb{O}]) \stackrel{\mathsf{pt}}{\to} \mathsf{Ult}(\mathsf{pt}(\mathcal{E}), \mathsf{Set})).\]

In full generality, we recover all objects by studying all geometric morphisms $\mathcal{E} \to \mathsf{Set}[\mathbb{O}]$, but those will not induce ultrafunctors. Yet, we can still trace a footprint of the ultrastructure of a geometric morphism as follows.

\begin{constr} \label{weakultra}
Let $f :\mathcal{E} \to \mathcal{D}$ be any geometric morphism between topoi in $\mathsf{RInj}(\mathsf{Emb}_\beta)$. Following the general theory of Kan injectivity, by the pseudocommutativity of the diagram on the left, we get a Beck--Chevalley-like natural transformation in the diagram on the right. 

\[\begin{tikzcd}
	{(\mathsf{Set}^X, \mathcal{E})} & {(\mathsf{Set}^X, \mathcal{D})} & {(\mathsf{Set}^X, \mathcal{E})} & {(\mathsf{Set}^X, \mathcal{D})} \\
	{(\mathsf{Sh}(\beta(X)), \mathcal{E})} & {(\mathsf{Sh}(\beta(X)), \mathcal{D})} & {(\mathsf{Sh}(\beta(X)), \mathcal{E})} & {(\mathsf{Sh}(\beta(X)), \mathcal{D})}
	\arrow["{i^\sharp_\mathcal{E}}", from=2-1, to=1-1]
	\arrow["{i^\sharp_\mathcal{D}}"', from=2-2, to=1-2]
	\arrow["{f^\bullet}", from=1-1, to=1-2]
	\arrow["{f^\bullet}"', from=2-1, to=2-2]
	\arrow["\cong"{description}, draw=none, from=2-1, to=1-2]
	\arrow["{f^\bullet}"', from=2-3, to=2-4]
	\arrow[from=1-4, to=2-4]
	\arrow["{i_\sharp^\mathcal{E}}"', from=1-3, to=2-3]
	\arrow["{i_\sharp^\mathcal{D}}", from=1-4, to=2-4]
	\arrow["{f^\bullet}", from=1-3, to=1-4]
	\arrow["\rho"{description}, Rightarrow, dashed, from=2-3, to=1-4]
\end{tikzcd}\]

Then, as it is observed in \cite{dlsolo}, $f$ is a morphism in $\mathsf{RInj}(\mathsf{Emb}_\beta)$ if and only if $\rho$ is an isomorphism. In full generality, though, we can go back to \Cref{emergentultrastructures} and look at the diagram below.

\[\begin{tikzcd}
	{(\mathsf{Set}^X, \mathcal{E})} && {(\mathsf{Set}^X, \mathcal{D})} \\
	\\
	{(\mathsf{Sh}(\beta(X)), \mathcal{E})} && {(\mathsf{Sh}(\beta(X)), \mathcal{D})} \\
	\\
	{(\mathsf{Set}^{\beta(X)}, \mathcal{E})} && {(\mathsf{Set}^{\beta(X)}, \mathcal{D})}
	\arrow["{q^\sharp_\mathcal{D}}", from=3-3, to=5-3]
	\arrow["{f^\bullet}", from=1-1, to=1-3]
	\arrow["{f^\bullet}"{description}, from=3-1, to=3-3]
	\arrow["{q^\sharp_\mathcal{E}}", from=3-1, to=5-1]
	\arrow["{f^\bullet}"', from=5-1, to=5-3]
	\arrow["{i_\sharp^\mathcal{E}}"', from=1-1, to=3-1]
	\arrow["{i_\sharp^\mathcal{E}}", from=1-3, to=3-3]
	\arrow["\rho"{description}, Rightarrow, from=3-1, to=1-3]
	\arrow["{\Sigma^\mathcal{E}_X}"', curve={height=50pt}, from=1-1, to=5-1]
	\arrow["{\Sigma_X^\mathcal{D}}", curve={height=-50pt}, from=1-3, to=5-3]
\end{tikzcd}\]

By the pasting $q^\sharp_\mathcal{D}(\rho)$, we obtain a natural transformation that compares the two constructions of ultraproducts,

\[q^\sharp_\mathcal{D}(\rho): f^\bullet \Sigma^\mathcal{E}_X \Rightarrow  \Sigma^\mathcal{D}_X f^\bullet.\]

\end{constr}

\begin{rem}[Left ultrafunctors]\label{def:weak} \label{representation}
One can take the couple $(f, \delta_\bullet) : \mathcal{A} \to \mathcal{B}$, consisting of a functor $f$ together with a natural transformation $\delta_X :f^\bullet \Sigma^\mathcal{E}_X \Rightarrow  \Sigma^\mathcal{D}_X f^\bullet$ and subject to some compatibility axioms as a definition of \textit{left} ultrafunctor between ultracategories. This would essentially amount to Lurie's notion of left ultrafunctor. This notion  allows for a representation of the whole topos $\mathcal{E}$, via a straightforward variation of the argument in \Cref{rep1}, \[\mathcal{E} \simeq \mathsf{Topoi}(\mathcal{E}, \mathsf{Set}[\mathbb{O}]) \to \mathcal{L}\mathsf{Ult}(\mathsf{pt}(\mathcal{E}), \mathsf{Set}).\] 
In this paper we shall not insist on left ultrafunctors.
\end{rem}



\subsection{Ultraprofiles in $\mathsf{Acc}^\mathsf{Sp}_\omega$}

In this subsection we explore a new arena of formal model theory, namely \textit{accessible profiles}. The notion is inspired by Marmolejo's treatment of Łoś categories, and by Lurie's stack-y interpretation of ultrastructures \cite[Sec. 4]{lurieultracategories}. Besides this motivational overlap though, the whole subsection offers original material. We shall offer a biadjunction between the $2$-category of accessible profiles and the $2$-category of topoi which enhances the Scott adjunction \cite{thcat} and define a notion of \textit{ultraobject} for this formal model theory.
\subsubsection{Accessible profiles}
\begin{defn}[Accessible profile]
An \textit{accessible profile} is pseudofunctor $\mathcal{A}^\bullet : \mathsf{Sp}^\circ \to \mathsf{Acc}_\omega$. We shall call $\mathsf{Acc}^\mathsf{Sp}_\omega$ the $2$-category of accessible profiles, pseudonatural transformations and modifications between those. We may omit the adjective \textit{accessible} from now on. A profile is \textit{modest} if it is a small weighted bicolimit of representables in the $2$-category of profiles in $\mathsf{Acc}^\mathsf{Sp}_\omega$. We denote $\underline{\mathsf{Acc}^\mathsf{Sp}_\omega}$ the full sub $2$-category of modest profiles.
\end{defn}

\begin{rem}[Representable profiles]
We have a Yoneda embedding,\[ \yo : \mathsf{Sp} \to \mathsf{Acc}^\mathsf{Sp}_\omega\] defined by $\yo(X)^Y = \mathsf{Sp}(Y,X)$. Indeed, when we equip $\mathsf{Sp}$ with the $2$-dimensional structure given by the specialization order, $\mathsf{Sp}(Y,X)$ always has directed colimits, and it is accessible because it is a small poset (with directed colimits). 
\end{rem}

\begin{constr}[The profile of models of a Grothendieck topos] \label{profileofmodels}
Let $\mathcal{E}$ be a Grothendieck topos, and define \[\mathbb{pt}(\mathcal{E})^\bullet := \mathsf{Topoi}(\mathsf{Sh}(\bullet), \mathcal{E}).\] $\mathbb{pt}(\mathcal{E})^\bullet$ clearly is a pseudofunctor $\mathsf{Sp}^\circ \to \mathsf{Acc}_\omega$, because indeed each $\mathsf{Topoi}(\mathsf{Sh}(X), \mathcal{E})$ is accessible and has directed colimits. This gives a $2$-functor, \[\mathbb{pt}^\bullet: \mathsf{Topoi} \to \mathsf{Acc}^\mathsf{Sp}_\omega.\]
\end{constr}

\begin{exa}[The accessible profile $\mathsf{Set}^\bullet$]
A very interesting topos to apply this construction to is the classifier of the theory of objects $\mathsf{Set}[\mathbb{O}]$, in this case we can explicitly compute $\mathbb{pt}^\bullet(\mathsf{Set}[\mathbb{O}])$ as follows, \[\mathbb{pt}^\bullet(\mathsf{Set}[\mathbb{O}]) = \mathsf{Topoi}(\mathsf{Sh}(\bullet), \mathsf{Set}[\mathbb{O}]) \simeq \mathsf{Sh}(\bullet).\]
In the rest of the text we will call this profile $\mathsf{Set}^\bullet$, to push the derivator-like intuition that any profile $\mathcal{A}^\bullet$ tells the story of its evaluation at $1$.
\end{exa}

\begin{rem}
There are several ways to look at functor $\mathbb{pt}^\bullet: \mathsf{Topoi} \to \mathsf{Acc}^\mathsf{Sp}_\omega$, we shall point out for the moment that it corresponds to the nerve of the $2$-functor $\mathsf{Sh}: \mathsf{Sp} \to \mathsf{Topoi}$ as depicted in the picture below.

\[\begin{tikzcd}
	& {\mathsf{Sp}} \\
	{\mathsf{Topoi}} && {\mathsf{Acc}_\omega^{\mathsf{Sp}}}
	\arrow["\yo", from=1-2, to=2-3]
	\arrow["{\mathsf{Sh}}"', from=1-2, to=2-1]
	\arrow["{\mathbb{pt}^\bullet}"', from=2-1, to=2-3]
\end{tikzcd}\]

Being a nerve functor, we feel optimistic that it could have a left adjoint, especially given that $\mathsf{Acc}_\omega^{\mathsf{Sp}}$ is almost a category of prestacks. Of course, as it often happens in category theory, \textit{the devil is in size issues}. We start with a preliminary result.
\end{rem}

\begin{prop}\label{yonedapointprofiles}
Let $X$ be a space such that $\yo(X)$ is an accessible profile.
For every Grothendieck topos $\mathcal E$, evaluation at $1_X$
induces an equivalence
\[
\begin{aligned}
  &\mathsf{Acc}_\omega^{\mathsf{Sp}}
    (\yo(X),\mathbb{pt}^\bullet(\mathcal E))\\
  &\qquad\simeq
    \mathsf{Topoi}(\mathsf{Sh}(X),\mathcal E).
\end{aligned}
\]
\end{prop}

\begin{proof}
By the $2$-Yoneda lemma, it suffices to check that, for every
$g:\mathsf{Sh}(X)\to\mathcal E$, the assignment
$f\mapsto g\mathsf{Sh}(f)$ preserves directed colimits.

Since $\yo(X)$ is an accessible profile, directed colimits
in $\mathsf{Sp}(Y,X)$ are computed pointwise. Observe
that every increasing net $(x_i)_{i\in I}$ converges to its
supremum. Thus, up to $T_0$-reflection, $X$ is a monotone convergence
space in the sense of \cite[Def.~II-3.12]{gierz2003continuous}.
For $f=\operatorname{colim}_i f_i$ in $\mathsf{Sp}(Y,X)$,
\cite[Lem.~II-3.13]{gierz2003continuous}, applied to the
characteristic maps of open subsets, gives
\[
  f^{-1}(U)=\bigcup_i f_i^{-1}(U)
  \qquad (U\subseteq X\text{ open}).
\]
Geometric morphisms between localic topoi are ordered
pointwise on inverse images of opens
\cite[IX.5, Prop.~2]{sheavesingeometry}, hence
  $\mathsf{Sh}(f)\cong
  \operatorname{colim}_i\mathsf{Sh}(f_i).$ 
\end{proof}

\begin{thm} \label{newadjunction}
The $2$-functor $\mathbb{pt}^\bullet$ has a (relative) left biadjoint $\Theta$, defined for all modest accessible profiles, i.e. for all modest profiles $\mathcal{A}^\bullet$ we have that $\mathsf{Topoi}(\Theta(\mathcal{A}^\bullet),\mathcal{E}) \simeq \mathsf{Acc}_\omega^{\mathsf{Sp}}(\mathcal{A}^\bullet, \mathbb{pt}^\bullet(\mathcal{E}))$.  

\[\begin{tikzcd}
	{\underline{\mathsf{Acc}^{\mathsf{Sp}}_\omega}} & {\mathsf{Topoi}} \\
	{\mathsf{Acc}^{\mathsf{Sp}}_\omega}
	\arrow[""{name=0, anchor=center, inner sep=0}, "\Theta", dashed, from=1-1, to=1-2]
	\arrow[hook', from=1-1, to=2-1]
	\arrow[""{name=1, anchor=center, inner sep=0}, "{\mathbb{pt}^\bullet}", from=1-2, to=2-1]
	\arrow["\dashv"{anchor=center, rotate=-89}, draw=none, from=0, to=1]
\end{tikzcd}\]

\end{thm}
\begin{proof}
Let $\mathcal A^\bullet$ be modest. Choose a small weighted
bicolimit presentation
\[
  \mathcal A^\bullet\simeq W\star(\yo D)
\]
in $\mathsf{Acc}_\omega^{\mathsf{Sp}}$, where $\mathcal J$
is a small $2$-category, $D:\mathcal J\to\mathsf{Sp}$ is a
diagram whose representables form a diagram of accessible
profiles, and $W:\mathcal J^\circ\to\mathsf{Cat}$ is a small
weight. Since $\mathsf{Topoi}$ admits small weighted
bicolimits, put
\[
  \Theta(\mathcal A^\bullet)
  :=W\star(\mathsf{Sh}\circ D).
\]

Write $(\mathcal J^\circ,\mathsf{Cat})$ for the $2$-category
of pseudofunctors, pseudonatural transformations and
modifications. For every Grothendieck topos $\mathcal E$,
the two bicolimit universal properties and
\Cref{yonedapointprofiles} give
\[
\begin{aligned}
  &\mathsf{Topoi}(\Theta(\mathcal A^\bullet),\mathcal E)\\
  &\quad\simeq
    (\mathcal J^\circ,\mathsf{Cat})\bigl(
      W,\mathsf{Topoi}(\mathsf{Sh}(D-),\mathcal E)\bigr)\\
  &\quad\simeq
    (\mathcal J^\circ,\mathsf{Cat})\bigl(
      W,\mathsf{Acc}_\omega^{\mathsf{Sp}}
      (\yo(D-),\mathbb{pt}^\bullet(\mathcal E))\bigr)\\
  &\quad\simeq
    \mathsf{Acc}_\omega^{\mathsf{Sp}}
    (\mathcal A^\bullet,\mathbb{pt}^\bullet(\mathcal E)).
\end{aligned}
\]
These equivalences are pseudonatural in $\mathcal E$.
Birepresentability therefore makes the construction independent,
up to equivalence, of the chosen presentation and supplies its
pseudofunctoriality in $\mathcal A^\bullet$. This gives the
required relative left biadjoint.

\end{proof}

\begin{rem}[The topos of coordinates of a modest profile] \label{coordinates}

The existence of $\Theta$ in the theorem above follow from the general theorem of biKan extensions and is computed by the formula \[\Theta = \mathsf{biLan}_{\yo} (\mathsf{Sh}).\]
Of course, this may seem as a very unsatisfactory equation, which is in principle hard to compute. Yet, because of the very structure of the $2$-category of topoi, we do have a way to compute $\Theta$ by the chain of equivalences below.
\begin{align*}
   \Theta (\mathcal{A}^\bullet) \simeq & \mathsf{Topoi} (\Theta (\mathcal{A}^\bullet), \mathsf{Set}[\mathbb{O}]) \\
        \simeq & \mathsf{Acc}_\omega^\mathsf{Sp}(\mathcal{A}^\bullet, \mathbb{pt}^\bullet(\mathsf{Set}[\mathbb{O}])) \\
       \simeq  & \mathsf{Acc}_\omega^\mathsf{Sp}(\mathcal{A}^\bullet, \mathsf{Set}^\bullet). 
\end{align*}
This remark gives an \textit{a posteriori} proof of the fact that the category $\mathsf{Acc}_\omega^\mathsf{Sp}(\mathcal{A}^\bullet, \mathsf{Set}^\bullet)$ is a Grothendieck topos for every modest profile. A priori, we have a forgetful functor $\mathsf{Acc}_\omega^\mathsf{Sp}(\mathcal{A}^\bullet, \mathsf{Set}^\bullet) \to \mathsf{Acc}_\omega(\mathcal{A}^1, \mathsf{Set})$. Conjecturally, it should be possible to show that such a functor preserve colimits and finite limits, and yields a geometric surjection $\mathsf{S}(\mathcal{A}^1) \twoheadrightarrow \mathbb{O}(\mathcal{A}^\bullet)$ where $\mathsf{S}(\mathcal{A}^1)$ is the Scott topos of $\mathcal{A}^1$ in the sense of \cite{thcat}, but this would involve (among the other things) a very delicate analysis to find a generator of $\mathbb{O}(\mathcal{A}^\bullet)$. Lurie does something similar to this in \cite[5.4.5]{lurieultracategories}.
\end{rem}

\begin{rem}[The Counit as a representation] \label{reprensentation}
Let $\mathcal{E}$ be a topos, and assume that $\mathbb{pt}^\bullet(\mathcal{E})$ is modest. Then, we have a counit $\epsilon: \Theta \mathbb{pt}(\mathcal{E}) \to \mathcal{E}$. Let us spell out explicitly the inverse image of such functor,
\[\epsilon_{\mathcal{E}}^*: \mathcal{E} \to \mathsf{Acc}_\omega^\mathsf{Sp}(\mathbb{pt}^\bullet(\mathcal{E}), \mathsf{Set}^\bullet).\]
The functor $\epsilon$ represents the usual \textit{test functor} to check whether we can recover the topos $\mathcal{E}$ from its \textit{(indexted) category of (topological) points}, similarly to \cite{thlo} or Makkai/Lurie's conceptual completeness. We expect every profile of the form $\mathbb{pt}^\bullet(\mathcal{E})$ to be modest, but we did not manage to prove such a statement.
\end{rem}

\subsubsection{Ultraprofiles}

\begin{defn}[Ultraprofile]
Let $i_X : X \to \beta(X)$ be the usual inclusion of a set $X$ in its space of ultrafilters. We say that a ultaprofile is a profile that is right Kan injective with respect to $\yo(i_X)$ for all $Xs$.
\end{defn}


\begin{prop} \label{ultaprofilebuono}
Let $\mathcal{E}$ be a topos. $\mathcal{E}$ is $\beta$-complete if and only if its profile of points $\mathbb{pt}^\bullet(\mathcal{E})$ is an ultraprofile.
\end{prop}
\begin{proof}
Follows directly from the observation that the map $i_X: \mathsf{Set}^X \to \mathsf{Sh}(\beta(X))$ is nothing but $\Theta(\yo(i_X))$ and the following chain of equivalences, \[\mathsf{Topoi}(\mathsf{Set}^X,\mathcal{E}) \simeq \mathsf{Topoi}(\Theta(\yo(X)),\mathcal{E}) \simeq \mathsf{Acc}_\omega^{\mathsf{Sp}}(\yo(X), \mathbb{pt}^\bullet(\mathcal{E})).\]
\end{proof}

\subsection{Ultraionads in $\mathsf{Bion}$}

The theory of ionads was introduced by Garner in \cite{ionads} and later developed in \cite{thgeo} and \cite{thlo}. Briefly, recall that a ionad is a category $C$ equipped with an lex comonad $\mathsf{Int}: \P(C) \to \P(C)$ on its category of small copresheaves that mimics the behavior of an interior operator. We refer to those papers for a proper introduction to the topic. In this subsection we shall discuss the notion of ultraionad and relate it to that of $\beta$-complete topoi. We start by recalling two useful results in the theory of ionads.

\begin{disc}[Ionads and spaces]
In \cite[Example 3.5(2)]{ionads}, Garner describes a general recipe to construct a buonded ionad $\Sigma X$ from a topological space $X$. If we look at $\mathsf{Sp}$ as a $2$-category, where the $2$-dimensional structure is given by the specialization order, this construction yields a $2$-adjunction,  \[ \Lambda : \mathsf{BIon} \leftrightarrows  \mathsf{Sp} : \Sigma, \]
which Garner describes in \cite[Example 4.8]{ionads}. Notice that the right adjoint is actually fully faithful, so that the $2$-category of spaces is a fully sub-$2$-category of bounded ionads.
\end{disc}

\begin{disc}[Ionads and topoi] \label{isbell}
In \cite[Sec. 3]{thgeo}, the author shows that there is a biadjunction between the $2$-category of topoi and the $2$-category of bounded ionads,
\[ \mathbb{O} : \mathsf{BIon}   \leftrightarrows  \mathsf{Topoi} :\mathbb{pt}. \]
This adjunction plays the same role of the adjunction relating topological spaces to locales, and indeed restricts to a biequivalence of $2$-categories between topoi with enough points and sober bounded ionads \cite[Sec 4.]{thgeo}.
\end{disc}


\begin{disc}[Clones of $\mathsf{Emb}_\beta$]
This whole paper is based on the study of the map $X \to \beta(X)$, and its instantiation in the $2$-categories of interests. We now collect three different instantiations of this morphism in the three $2$-categories that populate this section.
\begin{itemize}
    \item In $\mathsf{Sp}$, we find the class $\mathsf{emb}_\beta$, made of the maps $X \to \beta (X)$, where $X$ is a discrete set, and $\beta(X)$ is its space of ultrafilters. In a sense, we have been looking at this maps for the whole paper.
    \item In $\mathsf{BIon}$, we define $\mathbb{Emb}_\beta = \Sigma (\mathsf{emb}_\beta)$.
    \item In $\mathsf{Topoi}$, we have our class $\mathsf{Emb}_\beta$. It is easy to see that the following equations hold, \[\mathbb{O}(\mathbb{Emb}_\beta) \cong \mathsf{Emb}_\beta \quad \quad \mathbb{pt}(\mathsf{Emb}_\beta) = \mathbb{Emb}_\beta. \]
    
    \end{itemize}
\end{disc}

\begin{defn}[Ultraionads]
We define the $2$-category of ultraionads $\mathsf{UltBIon}$ to be $\mathsf{RInj}(\mathbb{Emb}_\beta)$.
\end{defn}

\begin{prop} \label{ionadsaregood}
\begin{itemize}
\item[]
    \item A topos $\mathcal{E}$ is $\beta$-complete if and only if its ionad of opoints $\mathbb{pt}(\mathcal{E})$ is an ultraionad.
    \item Let $(\mathcal{A}, \mathsf{Int})$ be a bounded ionad. If $\mathbb{O}(\mathcal{A})$ is a $\beta$-complete topos, then  $(\mathcal{A}, \mathsf{Int})$ is an ultraionad. If the ionad is sober, then the other implication is also true.
\end{itemize} 
\end{prop}
\begin{proof}
The first claim follows directly from the observation that $\mathbb{O}(\mathbb{Emb}_\beta) \cong \mathsf{Emb}_\beta$ and the properties of adjunctions. Similarly, the second claim follows from the observation that $\mathbb{pt}(\mathsf{Emb}_\beta) = \mathbb{Emb}_\beta$ and the properties of adjunctions.
\end{proof}


\section{Conclusions and open questions} \label{sec4}

We shed light on two different topics, namely the geometric nature/behavior of coherent topoi and what structures come into play in the very definition of ultrastructure and ultracategory. 
(Un)fortunately this paper opens many more doors than it closes and we shall comment on those in this section.

\subsection{Geometry of coherent topoi}

\Cref{sec1} is mostly devoted to set the diagram below. 
\[\begin{tikzcd}
	& {\mathsf{CohTopoi}} \\
	{\mathsf{RInj}(\mathsf{Emb}_\flat)} && {\mathsf{RInj}(\mathsf{Emb}_\beta)}
	\arrow["{\mathbb{U}_\beta^\flat}"{description}, from=2-1, to=2-3]
	\arrow["i"', hook', from=1-2, to=2-1]
	\arrow["j", hook, from=1-2, to=2-3]
\end{tikzcd}\]
    As we hinted in \Cref{characterisation} and \Cref{char2}, we do not consider this to be a definitive account on flat geometric morphisms and $\flat$-completeness. Let us list below some of the doors that this paper leaves open.
    
\begin{itemize}
    \item a general characterization, beyond the coherent case, of morphisms injective with respect to flat geometric morphisms.
    \item can we find a smaller and easier to understand class of flat geometric morphism whose saturation is the whole $\mathsf{Emb}_\flat$? 
    \item there is an interesting class of geometric morphisms that class that we haven't studied and seems a good candidate to saturate $\mathsf{Emb}_\flat$. This is the class given in \Cref{exampleindcompl} which sits in between $\mathsf{Emb}_\beta$ and $\mathsf{Emb}_\flat$. The relevance of this class for our theory is witnessed by \Cref{char2}.
    \item is $\mathsf{RInj}(\mathsf{Emb}_\beta)$ the completion of $\mathsf{CohTopoi}$ under coreflections?
    \item in this spirit of and referencing to the discussion in \Cref{characterisation}, a preliminary question to answer would be which topoi embed in a coherent one?
\end{itemize}

All these questions seem complelling to us, and would impact on a possible retuning of the notion of ultracategory.

\subsection{Formal model theories and ultraobjects}

On the side of formal model theory, even more questions are left to the reader. In a sense, the questions above are all preliminary to a good understanding of ultra-objects, as they shape our definitions. Yet, besides the questions about ultrastructures and ultracategories, other interesting topics emerge from this paper and would deserve some attention.

\begin{itemize}
    \item going back to \Cref{profileofmodels}, we wonder whether $\mathbb{pt}^\bullet$ is always modest. Some results obtained by Lurie on accessible ultrastructures seem to suggest that this is not such a strange conjecture. Of course, if $\mathbb{pt}^\bullet$ happens to be always modest, \Cref{newadjunction} ends up being an honest bi-adjunction.
    \item Speaking of \Cref{newadjunction}, is this biadjunction (pseudo)idempotent in the spirit of \cite{thgeo}?
    \item What is the relationship between the models of formal model theory we have discussed?
\end{itemize}

\section*{Acknowledgements}
I am indebted with \textit{Morgan Rogers} for all the discussions we had in the last year on a somewhat related project, those have informed and educated my knowledge of topos theory to a completely new level. I hope our collaboration will eventually see the light. I thank \textit{Axel Osmond} and \textit{Christian Espindola} for their inspiring comments on a preliminary draft of this paper. I am grateful to \textit{Peter Lumsdaine} for showing interest for this research in private conversations. I thank the anonimous referee for many useful comments that improved the exposition, and more specifically for suggesting the importance of discussing \Cref{hier}.
Finally I am indebted to \textit{Umberto Tarantino} and \textit{Lingyuan Ye}, I thank them for their friendship and for their continuous nudge and encouragement to update this paper during these years. A special thanks goes to Umberto, who finally forced me to do so.

\bibliography{thebib}
\bibliographystyle{alpha}
\end{document}